\documentclass[11pt]{article}
\usepackage{amssymb}
\usepackage{mathrsfs}
\usepackage{latexsym,amsmath,amssymb,amsfonts,epsfig,graphicx,cite,psfrag}
\usepackage{eepic,color,colordvi,amscd}
\usepackage{ebezier}
\usepackage{verbatim}
\usepackage{subcaption}
\usepackage{enumitem}
\usepackage{algorithm}
\usepackage{algorithmic}
\usepackage[normalem]{ulem}
\usepackage{hyperref}
\usepackage{amsthm}
\usepackage{longtable}
\usepackage{comment}
\usepackage{color}
\usepackage{xcolor}
\usepackage{graphicx} 
\usepackage{epstopdf}
\usepackage{longtable}
\usepackage{booktabs}
\usepackage{mathtools}
\usepackage{multicol, multirow}
\usepackage{capt-of}
\usepackage{longtable}
\usepackage{amsopn}
\usepackage{tikz}
\usetikzlibrary{arrows.meta, positioning}
\tikzset{
  block/.style={
    rectangle, 
    draw, 
    align=center,
    line width=1pt, 
    rounded corners, 
    text width=4cm, 
    font=\normalsize,
    inner ysep=6pt         
  },
  arrow/.style={-{Latex}, thick, black}
}
\usepackage{float}
\definecolor{blue}{rgb}{0,0,0.9}
\definecolor{red}{rgb}{0.9,0,0}
\definecolor{green}{rgb}{0,0.9,0}

\newcommand{\cT}{{\cal T}}

\newcommand{\cP}{{\cal P}}

\newcommand{\cA}{{\cal A}}

\newcommand{\bS}{{\mathbb{S}}}

\theoremstyle{plain}
\newtheorem{remark}{Remark}
\newtheorem{example}{Example}[section]

\newtheorem{assumption}{Assumption}
\newtheorem{definition}{Definition}

\newtheorem{theorem}{Theorem}

\newtheorem{lemma}{Lemma}

\def\Tr{\mathrm{Tr}}
\def\H{\mathcal{H}}
\def\<{\big\langle}
\def\>{\big\rangle}

\def\M{\mathcal{M}}
\def\P{\mathcal{P}}

\def\D{{\overline D}}

\def\A{\mathcal{A}}
\def\Q{\mathcal{Q}}
\def\B{\mathcal{B}}

\def\K{\mathcal{K}}
\def\L{\mathcal{L}}
\def\F{\mathcal{F}}

\def\R{\mathbb{R}}
\def\bN{\mathcal{N}}

\def\S{\mathbb{S}}

\def\O{\mathcal{O}}

\def\hS{\widehat{S}}

\def\tb#1{\textbf{#1}}

\def\wt{\widetilde}

\def\qRq{\quad\longrightarrow\quad}

\def\tol{\tt{tol}}
\def\timelimit{\tt{TimeLimit}}
\def\cA{\mathcal{A}}
\def\alp{\alpha}

\def\qLRq{{\quad \Longleftrightarrow \quad}}
\def\qRq{{\quad \Longrightarrow \quad}}
\def\tL{{\tilde{\mathcal{L}}}}
\def\Kj{{\K^{(j)}}}
\def\Yj{{Y^{(j)}}}

\def\Mj{{\mathcal{M}_{h_j}}}
\def\rank{{\operatorname{rank}}}
\let\svthefootnote\thefootnote
\newcommand\blankfootnote[1]{
\let\thefootnote\svthefootnote
}
\usepackage[many]{tcolorbox}
\newtcolorbox{boxA}{
    fontupper = \bf,
    boxrule = 1.5pt,
    colframe = black 
}
\newcommand{\xdownarrow}[1]{%
  {\left\downarrow\vbox to #1{}\right.\kern-\nulldelimiterspace}
}

\newcommand{\barn}{\bar{n}}

\def\inprod#1#2{\langle#1,\,#2\rangle}
\def\degg#1{\lceil#1\rceil}
\def\rmt{{\rm t}}

\begin{document}
        \title{
        A Low-rank Augmented Lagrangian Method for Polyhedral–SDP and Moment-SOS Relaxations of Polynomial Optimization
        }
        \author{
         Di Hou\thanks{Department of Mathematics, National
          University of Singapore, Singapore
          119076 ({\tt dihou@u.nus.edu}).
          }, \quad 
	 Tianyun Tang\thanks{Institute of Operations Research and Analytics, National
          University of Singapore, Singapore
          119076 ({\tt ttang@u.nus.edu}).
          }, \quad 
	   Kim-Chuan Toh\thanks{Department of Mathematics, and 
          Institute of Operations Research and Analytics, 
          National University of Singapore, 
          Singapore
          119076 ({\tt mattohkc@nus.edu.sg}). 
         The research of this author is supported by the Ministry of Education, Singapore, under its Academic Research Fund Tier 3 grant call (MOE-2019-T3-1-010).
         }
 	}
	\date{August 29, 2025}
\maketitle


\begin{abstract}
{
Polynomial optimization problems (POPs) can be reformulated as geometric convex conic programs, as shown by Kim, Kojima, and Toh (SIOPT 30:1251–1273, 2020), though such formulations remain NP-hard. 
In this work, we prove that several well-known relaxations can be unified under a common polyhedral–SDP framework, which arises by approximating the intractable cone by tractable intersections of polyhedral cones
with the positive semidefinite matrix cone. Although effective in providing tight lower bounds, these relaxations become computationally expensive as the number of variables and constraints grows at the rate of $\Omega(n^{2\tau})$ with the relaxation order $\tau$.
To address this challenge, we propose RiNNAL-POP, a low-rank augmented Lagrangian method (ALM) tailored to solve large-scale polyhedral–SDP relaxations of POPs. 
To efficiently handle the $\Omega(n^{2\tau})$ nonnegativity and consistency constraints, we design a tailored projection scheme whose computational cost scales linearly with the number of variables.
In addition, we identify a hidden facial structure in 
the polyhedral--SDP relaxation, which enables us to eliminate  a large number of linear constraints by restricting the matrix variable to affine subspaces corresponding to exposed faces of the semidefinite cone.
The latter enables us to efficiently solve the factorized ALM subproblems 
 over the affine subspaces.
At each ALM iteration, 
we additionally carry out a single projected gradient step with respect to the original matrix variable to automatically adjust the rank and escape from spurious local minima when necessary.
We also extend our RiNNAL-POP algorithmic framework to solve moment-SOS relaxations of POPs.
Extensive numerical experiments on various benchmark problems demonstrate the robustness and efficiency of RiNNAL-POP in solving large-scale polyhedral–SDP relaxations. 
}

\end{abstract}

\bigskip
\noindent{\bf keywords:} semidefinite programming, augmented Lagrangian, polynomial optimization
\\[5pt]

\section{Introduction}\label{sec-intro}

\subsection{Polynomial optimization}

In this paper, we consider the following polynomial optimization problem:
\begin{equation}\label{eq-POP}
\zeta^*=\min \Big\{ f_0(w) \mid  f_i(w) = 0,\; i\in [m],\; w\in D \Big\},\tag{POP}
\end{equation}
where $D\subseteq\R^n$ is a cone and each $f_i(w)$ is a real-valued multivariate polynomial.
\eqref{eq-POP} provides a unified framework for nonconvex polynomial optimization: taking $D=\R^n$ recovers equality‐constrained problems, taking $D=\R_+^n$ recovers nonnegatively‐constrained problems, and as we will show in Subsection~\ref{subsec-ineq}, general inequality‐constrained problems can be handled by introducing standard or squared‐slack variables.

POPs arise frequently in combinatorial optimization, control theory, signal processing, and engineering design. Unfortunately, they are NP‐hard to solve in general. To address this challenge, various convex reformulations have been proposed to approximate their global minimizers \cite{kim2020geometrical,burer2009copositive,lasserre2001global}.

In the next subsection, we present a unified convex conic reformulation of \eqref{eq-POP}, which encompasses both completely positive programming (CPP)
and positive semidefinite programming (SDP).
To keep the presentation concise, we focus mainly on purely conic domains, in particular the nonnegative orthant $D=\R^n_+$ and the Euclidean space $D=\R^n$.
But we should emphasize that the
theoretical analysis (and the algorithm introduced in Section~\ref{sec-alg}) also applies to relaxations of \eqref{eq-POP} even when $D$ is a general semialgebraic–conic domain (cf.\ \cite{kim2020geometrical}).

\subsection{Convex conic reformulation}
To derive a tractable convex reformulation of the nonconvex problem \eqref{eq-POP}, we lift it into a higher-dimensional matrix space in three steps: (i) homogenize all polynomials to a common even degree, (ii) introduce a matrix variable whose entries collect all monomials so that each polynomial constraint becomes linear, and (iii) relax by dropping the rank‐one requirement and replacing the resulting set of matrices by its convex hull. Under mild geometric conditions, the resulting formulation is equivalent to the original problem.

\medskip
\noindent\tb{Homogenization.}
For any positive integer $\tau\ge\max\{\lceil\deg f_i/2\rceil\mid i=0,\dots,m\}$, problem \eqref{eq-POP} is equivalent to its degree-$2\tau$ homogeneous form:
\begin{equation}\label{eq-POP-homo}
    \zeta^*  =  \min \Big\{ \bar{f}_0(x) \mid  \bar{f}_i(x) = 0, i \in [m],\; x_0 = 1,\;  x = (x_0,w)\in \D \Big\},
\end{equation}
where $\D:=\R_+\times D\subseteq \R^{n+1}$ and $\bar{f}_i$ denotes the degree-$2\tau$ homogenization of $f_i$ (see the notation in \eqref{eq-notation-homo} of Subsection \ref{subsec-notations}).

\medskip
\noindent\tb{Lifting.}
Let $ \cA_\tau = \{ \alp = (\alp_0,\ldots, \alp_n) \in \mathbb{N}^{n+1} \mid |\alp| = \tau\}$ be the set of all degree-$\tau$ monomial exponents in $\R^{n+1}$. 
For any exponent $\alpha\in \mathbb{N}^{n+1}$, and $x=(x_0;x_1;\ldots;x_n)\in \R^{n+1}$, we define $x^\alpha = x_0^{\alpha_0}x_1^{\alpha_1}\cdots x_n^{\alpha_n}$.
Choose a subset $\cA\subseteq \cA_\tau$ so that $\cA\times \cA$ contains every monomial exponents appearing in $\bar{f}_i(x)$ for $i=0,\dots,m$. 
Define $ u^\mathcal{A}(x)=(x^\alpha)_{\alpha\in\cA}$ as
the column vector of monomial basis with exponents in $\cA$. 
Let $\bS^\cA$ be the set of symmetric matrices with rows and columns
indexed by $\cA$.
Then each $\bar{f}_i(x)$ can be written as
\[
\bar{f}_i(x) = \left\langle{Q^i},{u^\cA(x) \left(u^\cA(x)\right)^\top}\right\rangle \;\mbox{for some}\;
  Q^i \in \bS^\cA.
\]
Consequently, problem~\eqref{eq-POP-homo} is equivalent to the nonconvex conic program
\begin{equation}\label{eq-COP-nonconvex}
    \eqref{eq-POP-homo}\equiv{\rm COP}(K\cap J):\quad\zeta^*=\min\Big\{\inprod{Q^0}{X} \mid X \in K\cap J, \; \inprod{H^0}{X}=1\Big\},
\end{equation}
where $H^0 \in \mathbb{S}^\mathcal{A}$ has zero entries everywhere except in the row and column indexed by the exponent $e_1^\tau := (\tau, 0, \ldots, 0) \in \mathcal{A}$ with $H^0_{e_1^\tau, e_1^\tau} = 1$, and
\begin{align} \label{eq-K}
K = \Big\{ u^\cA(x) \left(u^\cA(x)\right)^\top \mid x \in \D
\Big\},\quad 
  J = \Big\{ X\in {\rm co}(K) \mid \inprod{Q^i}{X} = 0\;\forall\; i\in [m]
  \Big\}.
\end{align}
In the above, ${\rm co}(K)$ denotes the convex hull of $K$.

\medskip
\noindent\tb{Convex relaxation.}
Since $K$ is generally nonconvex, solving ${\rm COP}(K \cap J)$ directly is difficult. Instead, we consider the convex relaxation problem over the convex set $J$:
\begin{equation}\label{eq-COP-convex}
    {\rm COP}( J):\quad \zeta^J=\min\Big\{\inprod{Q^0}{X} \mid X \in  J, \; \inprod{H^0}{X}=1\Big\}.
\end{equation}
This relaxation is not only convex, but under appropriate conditions, it is also tight, that is, it yields the same optimal value as the original nonconvex problem ${\rm COP}(K \cap J)$ \eqref{eq-COP-nonconvex}. The following result provides sufficient conditions for this equivalence (see \cite[Theorem 3.1]{kim2020geometrical} and 
\cite[Corollary 2.2]{arima2024further} for details).
\begin{theorem}\label{thm-cop-equiv}
    Suppose the following conditions hold: \tb{(A)} $K$ is a nonempty cone in $\bS^\cA$; \tb{(B)} $J$ satisfies ${\rm co}(K \cap J) = J$.
    Then 
    \begin{equation*}
        -\infty<\zeta^{J}<\infty\quad \Longleftrightarrow \quad 
    -\infty<\zeta^{J}=\zeta^*<\infty.
    \end{equation*}
\end{theorem}
\begin{remark}
    The geometric condition ${\rm co}(K \cap J) = J$ plays a crucial role in ensuring the exactness of the convex reformulation. This condition has been further studied in \cite{kim2020geometrical, arima2024further}, where its implications and sufficient conditions are examined in detail. In particular, a sufficient condition for Condition \textbf{(B)} to hold is that $J$ is a face of ${\rm co}(K)$.
\end{remark}

\begin{remark}\label{remark-cone}
When $\tau = 1$, $\cA=\cA_1$ and $u(x) = (x_0; \ldots; x_n)$, the convex conic reformulation ${\rm COP}(J)$ specializes to two well-known cases, depending on the choice of $D$:
    \[
    \begin{aligned}
& D = \mathbb{R}^n 
&\Longrightarrow\; 
\D &= \mathbb{R}_+\times\mathbb{R}^n
&\Longrightarrow\;& 
\;\mathrm{co}(K) = \mathbb{S}_+^{n+1}, \quad &\text{(SDP)}
\\
& D = \mathbb{R}_+^n 
&\Longrightarrow\; 
\D &= \mathbb{R}_+^{n+1}
&\Longrightarrow\;& 
\;\mathrm{co}(K) = \mathrm{CPP}^{n+1},\quad &\text{(CPP)}
\end{aligned}
    \]
    where $\mathrm{CPP}^{n+1} = \mathrm{co}\left\{xx^\top \mid x \in \mathbb{R}_+^{n+1} \right\}$ is the cone of completely positive matrices~\cite{berman2003completely}.
\end{remark}

\subsection{Polyhedral-SDP relaxation}

Although ${\rm COP}(J)$ serves as a convex reformulation of the original nonconvex polynomial optimization \eqref{eq-POP}, it may still be intractable because ${\rm co}(K)$ can be very complex, where $K$ is 
defined in \eqref{eq-K}. For example, testing membership in the cone $\mathrm{CPP}^{n+1}$ is NP‐hard in general \cite{dickinson2014computational}. In this subsection, we introduce tractable relaxations of ${\rm COP}(J)$ that yield computationally tractable lower bounds for $\zeta^*$.

To obtain a tractable approximation, we replace ${\rm co}(K)$ by the intersection of three cones: the positive semidefinite (PSD) cone $\mathbb{S}_+^\mathcal{A}$, a problem-dependent polyhedral cone $\mathcal{P}^\mathcal{A} \subseteq \mathbb{S}^\mathcal{A}$, and the consistency cone:
\begin{equation*}
\mathcal{L}^{\mathcal{A}} = \Big\{ X \in \mathbb{S}^{\mathcal{A}} \mid X_{\alpha,\beta} = X_{\gamma,\delta} \;\text{if } \alpha + \beta = \gamma + \delta \Big\}.
\end{equation*}
The equality constraints in $\mathcal{L}^\A$ are generated
from the fact that $x^{\alpha+\beta}=x^{\gamma+\delta}$
if $\alpha+\beta=\gamma+\delta.$

As we shall show in Section~\ref{subsec-FR}, many of the equality constraints defining $J$ exhibit a facial structure and can be compactly written as $AX = 0$ for some full-row-rank matrix $A\in\R^{m\times \cA}$. In this paper, we exploit this structure and consider the following polyhedral-SDP relaxation:
\begin{equation}\label{eq-SDP}
\zeta^* \geq \zeta^{\rm relax} 
= \min \left\{ \langle Q^0, X \rangle \;\middle|\;
\begin{aligned}
&\langle H^0, X \rangle = 1, \quad
AX = 0, \;
\langle Q^i, X \rangle = 0 \; \forall i \in [\ell], \\
&X \in \mathbb{S}_+^\mathcal{A} \cap \mathcal{P}^\mathcal{A} \cap \mathcal{L}^\mathcal{A}
\end{aligned}
\right\}.\tag{P}
\end{equation}
The integer $\tau$ such that $\A\subseteq \A_\tau$ is called the relaxation order.
This formulation covers a range of convex relaxations—for example, taking
\begin{align*}
        \text{(SDP)}\quad&\quad \mathrm{co}(K) = \mathbb{S}_+^{n+1} 
        \hspace{27pt}\Longrightarrow\quad
        \mathcal{P}^\mathcal{A}=\bS^{\cA},
        \\
        \text{(DNN)}\quad&\quad \mathrm{co}(K) = {\rm CPP}^{n+1}
        \quad \Longrightarrow \quad
        \mathcal{P}^\mathcal{A}=\bN^{\cA},
    \end{align*}
yields the standard SDP and doubly‐nonnegative (DNN) relaxations, respectively. Here $\bN^\cA$ denotes the
cone of nonnegative matrices in $\S^\cA$. Moreover, by selecting \(\cP^{\cA}\) appropriately, one can obtain relaxations for partial‐nonnegativity constraints and other structured domains. If the original polynomial optimization problem~\eqref{eq-POP} is already homogeneous on a conic domain, one may skip the homogenization step and obtain the relaxation problem~\eqref{eq-SDP} without the fixed constraint $\langle H^0,X\rangle=1$. 

To illustrate our framework, we present Example~\ref{example-pop} 
in Appendix~\ref{sec-A}
the quadratic POP from \cite[Example 2]{pena2015completely} as an example to demonstrate its CPP reformulation, DNN relaxation, and facial structure.

\subsection{Low-rank augmented Lagrangian method (ALM)}
Our main question is how to efficiently solve the relaxed problem \eqref{eq-SDP}, which includes both SDP and DNN relaxations as special cases. For the ease of discussion, 
we focus on the case $\cA = \cA_\tau$, i.e., the set of all 
exponents with degrees equal to $\tau$.
We identify four major challenges in solving \eqref{eq-SDP}:
\begin{enumerate}
    \item A huge number of $\Omega(mn^\tau +\ell)$ equality constraints; 
    \item A huge number of $\Omega(n^{2\tau})$ nonnegativity constraints when $D=\R^n_+$;
    \item A huge number of $\Omega(n^{2\tau})$ consistency constraints for fixed $\tau\geq 2$;
    \item Failure of Slater’s condition and strong duality whenever $A \neq 0$.
\end{enumerate}

In this paper, we propose RiNNAL-POP, a low-rank augmented Lagrangian method for solving the general polyhedral-SDP relaxation \eqref{eq-SDP}. Our approach builds upon the key ideas of RiNNAL~\cite{RiNNAL} and RiNNAL+~\cite{RiNNAL+}, but extends them to handle broader constraint structures and introduces new techniques for exploiting consistency and facial structure within a unified low-rank framework.
RiNNAL-POP employs a hybrid strategy that alternates between two phases to efficiently solve the following ALM subproblem:
\begin{equation}\label{eq-ALMsub-CVX}
    \min\left\{
    \begin{aligned}
        \langle Q^0,X\rangle 
        &+ \frac{\sigma}{2} \|\sigma^{-1}y-\Q(X)\|^2 \\
        &+ \frac{\sigma}{2}\|X-\sigma^{-1}W-\Pi_{{\P}}(X-\sigma^{-1}W)\|^2
    \end{aligned}
    \;\middle|\; AX=0,\; X\in \bS^\cA_+ \right\}, \tag{CVX}
\end{equation}
where $\sigma > 0$ is the penalty parameter, $y$ and $W$ are the Lagrange multipliers for the constraints $\Q(X) = 0$ and $X \in \mathcal{P}$, respectively, and
\begin{equation*}
    \Q(X):=\left[\langle Q^1,X\rangle,\dots,\langle Q^\ell,X\rangle\right]^\top,
    \quad 
    \P:=\Big\{X\in \P^\cA\cap \L^\cA
    \mid \langle H^0, X \rangle = 1
    \Big\}.
\end{equation*}
RiNNAL-POP alternates between a low‐rank phase and a convex‐lifting phase to solve \eqref{eq-ALMsub-CVX} efficiently.

\medskip
\noindent\textbf{Low-rank phase.} 
Suppose the ALM subproblem \eqref{eq-ALMsub-CVX} admits an optimal solution of rank $r$. To exploit this potential low-rank structure, we apply the Burer--Monteiro (BM) factorization $X = RR^\top$ with $R \in \mathbb{R}^{\cA \times r}$, and rewrite \eqref{eq-ALMsub-CVX} as the following nonconvex problem:
\begin{equation}\label{eq-ALMsub-LR}
    \min\left\{ 
    \begin{aligned}
       & \langle Q^0,RR^\top\rangle 
        + \frac{\sigma}{2} \|\sigma^{-1}y-\Q(RR^\top)\|^2 \\
        &+ \frac{\sigma}{2}\|RR^\top-\sigma^{-1}W-\Pi_{{\P}}(RR^\top-\sigma^{-1}W)\|^2
    \end{aligned}
    \;\middle|\;
    AR=0, \; R\in\R^{\A\times r} \right\}.\tag{LR}
\end{equation}
We solve this nonconvex subproblem by using the projected gradient (PG) method on $R$.

\medskip

\noindent\textbf{Convex-lifting phase.}  
Since the low-rank subproblem \eqref{eq-ALMsub-LR} is nonconvex, the iterate $R_t$ may converge to a suboptimal stationary point. Moreover, the appropriate rank is not known in advance and typically requires careful tuning. To address both issues, the algorithm switches to a convex-lifting phase once $R_t$ reaches near-stationarity. In this phase, we perform a PG step on the original convex problem \eqref{eq-ALMsub-CVX}, initializing at $X_t = R_t R_t^\top$. We then factorize the updated matrix $X_{t+1}$ to obtain a new iterate $R_{t+1}$, which serves as the starting point for the next low-rank phase.

This lifting step guarantees a monotonic decrease in the objective value of the ALM subproblem and helps ensuring the convergence to the global optimal solution of the ALM subproblem. Unlike rank-truncation heuristics—which may increase the function value when discarding small singular values -- the lifting step automatically adjusts the rank without requiring manual tuning. The idea of using a PG step to automatically update the rank has been explored in prior works such as \cite{lee2022escaping,RiNNAL+}. In particular, the numerical experiments in \cite{RiNNAL+} demonstrate that this approach performs remarkably well and typically identifies the correct rank after only a few PG steps. In this paper, we provide a theoretical justification for this phenomenon: we prove that the PG iterates provably identify the optimal rank via manifold identification techniques. Moreover, as we will show in Subsection~\ref{subsec-CVX-phase}, the projection required in each PG step has a closed-form expression, and the projection onto the PSD cone involves only a single eigenvalue decomposition, making each step computationally affordable.

\subsection{Summary of our contributions}

Our paper’s contributions are summarized as follows:
\begin{enumerate}
    \item We revisit the polyhedral–SDP relaxation framework of \cite{kim2020geometrical} and explicitly provide its equivalence and embedding relationships with RLT, SDP, DNN, and moment–SOS relaxations for general nonconvex polynomial optimization problems with equality and inequality constraints. 
    \item We leverage the facial structure present in broad families of linear constraints (with RLT and moment–SOS relaxations as two key examples) that expose a face of the semidefinite cone via the equality constraint $AX=0$.
    While this structure destroys Slater’s condition and impedes strong duality and numerical stability, our framework overcomes these drawbacks by (1) absorbing the face into the Burer–Monteiro factorization through the smooth subspace constraint $AR=0$ and (2) fully restoring dual feasibility by explicitly recovering all dual variables from any primal solution.
    \item We design RiNNAL-POP, a two-phase augmented Lagrangian method for solving the polyhedral–SDP relaxation \eqref{eq-SDP} of general polynomial optimization problems. In phase one, it solves a low-rank ALM subproblem on the subspace $AR=0$. In phase two, it applies a single projected gradient update in the full matrix space, where the projection onto the PSD cone with affine constraints admits an explicit closed-form expression. RiNNAL-POP computes the projection onto the intersection of the polyhedral cone and the consistency cone efficiently by averaging identical monomial entries, and enjoys global convergence under mild conditions. Through manifold identification, we prove that the optimal rank is identified automatically, removing the need for manual rank tuning.
    \item We extend RiNNAL‑POP to solve general $\tau$th‑order moment–SOS relaxations by recasting the problem in a split augmented‑Lagrangian framework with auxiliary variables for the PSD localization matrices.
    \item  We conduct extensive computational experiments to evaluate the performance of our RiNNAL-POP algorithm for solving the polyhedral–SDP  relaxation \eqref{eq-SDP} and moment-SOS relaxations of various polynomial optimization problems.
\end{enumerate}

\subsection{Organization}
This paper is organized as follows. In Section~\ref{sec-relaxation}, we present our unified polyhedral–SDP framework, show how RLT, SDP, DNN, and moment–SOS relaxations fit within this framework, and examine the hidden facial structure. In Section~\ref{sec-alg}, we introduce the augmented Lagrangian framework with the low-rank factorization for solving \eqref{eq-SDP}. 
In Section~\ref{sec-extension}, we present the extension of the algorithm
in Section~\ref{sec-alg} to solve moment-SOS relaxation problems.
Section \ref{sec-exp} presents several experiments to demonstrate the efficiency and extensibility of the proposed method. Finally, we conclude the paper in Section \ref{sec-conclusion}.

\subsection{Related works}

Interior-point methods (IPMs), as implemented in solvers such as SDPT3 \cite{TTT}, SeDuMi \cite{sturm1999using}, and DSDP \cite{benson2008algorithm},
can be applied to solve \eqref{eq-SDP}. While they 
are highly effective for small to medium-size SDPs,
they scale poorly for large SDP problems, especially for 
the problem \eqref{eq-SDP} due to their $\mathcal{O}(n^{6\tau})$ per-iteration complexity and high memory demands.
To improve scalability, first-order methods based on the alternating direction method of multipliers (ADMM) and the augmented Lagrangian method (ALM) have been widely adopted. For example, SDPNAL+ \cite{SDPNAL,SDPNALp1,SDPNALp2} has shown strong performance for SDP and DNN problems with matrix variables of moderate dimensions. 
Low‐rank solvers that factorize $X = RR^\top$ \cite{wang2023decomposition,wang2023solving,monteiro2024low,tang2024feasible,han2024low} can be quite effective when the solution is low rank and the number of constraints is moderate. 
Several low-rank solvers have also been proposed for SDPs \cite{wang2023decomposition,wang2023solving,monteiro2024low,tang2024feasible,han2024low}, based on factorizing the matrix variable as $X = RR^\top$ to reduce dimensionality. These methods are effective when the number of constraints is moderate and the solution is of low rank.
However, all the methods mentioned above face significant challenges on large-scale instances—particularly DNN relaxations—due to the rapidly growing number of equality and consistency constraints. Moreover, they may also suffer from the failure of strong duality and do not exploit the possible facial structure of the feasible set.

A closely related line of work is our previous development of RiNNAL \cite{RiNNAL} and RiNNAL+ \cite{RiNNAL+}, which are Riemannian ALM for DNN and SDP‐RLT relaxations of mixed‐binary quadratic programs. These solvers leverage the low-rank structure of the solution and penalize general inequality and equality constraints while preserving equalities that define the PSD cone’s reduced face and enforce binary constraints in the ALM subproblem. After low-rank factorization, the resulting feasible region of the ALM subproblem becomes an algebraic variety with favorable geometric properties. RiNNAL+ outperforms other solvers on large‐scale DNN and SDP‐RLT problems, but they are limited to mixed‐binary quadratic programs, which do not include the consistency constraints corresponding to $\mathcal{L}^\mathcal{A}$ 
in a general polynomial optimization problem.

\subsection{Notation}\label{subsec-notations}

For any set \(S\), we denote 
its convex hull by $\operatorname{co}(S)$.
Fix a nonnegative integer $\tau$.  We define
\[
\cA_\tau \;=\;\bigl\{\alpha=(\alpha_0,\alpha_1,\dots,\alpha_n)\in\mathbb{N}^{n+1}:\;|\alpha|=\alpha_0+\alpha_1+\cdots+\alpha_n=\tau\bigr\}
\]
to be the set of all degree‐$\tau$ monomial exponents in $\mathbb{R}^{n+1}$ with cardinality 
$
\barn_\tau =\binom{n+\tau}{\tau}.
$
Let $e_1^\tau=(\tau,0,\dots,0)\in\cA_\tau$ denote the multi‐index with $\alpha_0=\tau$ and all other components zero.
We write
$
[w]_\tau \;\in\;\mathbb{R}^{\barn_\tau}
$
for the column vector whose entries are all monomials in $w=(w_1,\dots,w_n)$ of degree $\le\tau$, arranged in some fixed ordering.
Let $p:\mathbb{R}^n\to\mathbb{R}$ be any polynomial such that 
$p(w) = \sum_{\alpha} p_\alpha w^\alpha$. For $\tau\ge\bigl\lceil\deg(p)/2\bigr\rceil$, the homogenization of $p$ of degree $2\tau$ is defined by
\begin{equation}\label{eq-notation-homo}
\bar p(x)\;=\;\bar p(x_0,w)\;=\; x_0^{2\tau}\; p\Bigl(\frac{w}{x_0}\Bigr) \;=\;
\sum_{\alpha} p_\alpha w^\alpha x_0^{2\tau-|\alpha|},
\qquad x=(x_0,w)\in\mathbb{R}^{n+1}.
\end{equation}
By construction, $\bar p$ is homogeneous of degree $2\tau$, and in particular satisfies $\bar p(1,w)=p(w)$.
Whenever $\cA\subseteq \cA_\tau$ is a subset of cardinality $|\cA|$, we write
\[
\bS^{\cA}\;=\;\bS^{|\cA|},
\qquad
\mathbb{R}^{\cA}\;=\;\mathbb{R}^{|\cA|},
\qquad 
\]
to emphasize that the matrices in $\bS^{|\cA|}$ and vectors in $\mathbb{R}^{|\cA|}$
are indexed by $\cA$.

\section{Polyhedral-SDP relaxations}\label{sec-relaxation}

In this section, we present a unified framework for constructing polyhedral‐SDP relaxations of polynomial optimization problems. We begin by showing how any set of polynomial inequalities can be rewritten as equalities via either standard slack variables or squared‐slack variables, thereby casting the problem in the general form of \eqref{eq-POP} and enabling its direct polyhedral-SDP relaxation \eqref{eq-SDP}. We then demonstrate that the classical moment–SOS hierarchy yields precisely the same polyhedral‐SDP relaxation, so that each $\tau$th‐order moment–SOS program can be interpreted as a facially‐exposed SDP of the form \eqref{eq-SDP}. Next, we explore how linear constraints of the form $AX=0$ (arising, for example, from RLT and or moment–SOS relaxations) expose a proper face of the semidefinite cone $\bS_+^\cA$, destroying Slater’s condition and giving rise to degeneracies in standard solvers.

\subsection{Inequality reformulation via slack variables}\label{subsec-ineq}

Many polynomial optimization problems include inequality constraints. For clarity, we focus here on the case without equality constraints because extending to problems with both equalities and inequalities is straightforward. 

Below, we describe two equivalent reformulations that convert each inequality into an equality, so that the problem can be written in the form of \eqref{eq-POP} and then relaxed to \eqref{eq-SDP}. The overall procedure is summarized in Figure~\ref{fig:reformulation}.
Let $\cA \subseteq \cA_\tau$ be a monomial index set chosen so that $\cA \times \cA$ contains all monomial exponents appearing in the homogenized polynomials $\bar f_i(x)$ for $i = 0,\dots,m$. We define two extended index sets:
\[
\cA_1 \;=\; (\cA \times \{0\}) \;\cup\; \bigl(e_1^{\tau - 1} \times \{1\}\bigr),
\qquad
\cA_2 \;=\; (\cA \times \{0\}) \;\cup\; \bigl(e_1^{\tau - 2} \times \{2\}\bigr),
\]
which correspond, respectively, to the standard‐slack and squared‐slack formulations. If the same polyhedral set $\mathcal{P}^{\cA_1} = \mathcal{P}^{\cA_2}$ is imposed in both cases, the resulting polyhedral‐SDP relaxations are identical. Verifying this equivalence amounts to a straightforward algebraic check and is omitted for brevity.

In practice, the choice between these two reformulations depends on the domain $D$. When $D = \mathbb{R}^n$, one typically uses the squared‐slack reformulation corresponding to $\cA_2$, which leads directly to an SDP relaxation. When $D = \mathbb{R}^n_+$, the standard‐slack reformulation corresponding to $\cA_1$ is more appropriate, yielding a DNN relaxation.

\begin{center}
\begin{figure}[ht]
\begin{tikzpicture}[node distance=1.5cm and -4cm]
  \node[block,text width=8cm] (prob) {
    $\displaystyle \min \Big\{ f_0(w) \mid   f_i(w)\geq 0,\; i\in [m] ,\; w\in D\Big\}$
  };
  \node[block, below left=of prob,text width=7cm]  (slack) {
  Reformulation via \\
  \tb{standard slack variables}\\[-3pt]
  \noindent\rule{\linewidth}{1pt}
  \begin{equation*}\label{eq-ieq-slack}
    \min \left\{ f_0(w) \;\middle|\;
    \begin{aligned}
        &f_i(w)-u_i= 0,\; i\in [m],\\
        &(w,u)\in D\times\R^{m}_+
    \end{aligned}
    \right\}
\end{equation*}
  };
  \node[block, below right=of prob,text width=7cm] (sq)    
  {
  Reformulation via \\
  \tb{squared slack variables}\\[-3pt]
  \noindent\rule{\linewidth}{1pt}
  \begin{equation*}
    \min \left\{ f_0(w) \;\middle|\;
    \begin{aligned}
        &f_i(w)-v_i^2= 0,\; i\in [m] ,\\
        &(w,v)\in D\times \R^{m}
    \end{aligned}
    \right\}
\end{equation*}
  };;
  \node[block, below=of slack,text width=5cm]    (mom1)  
  {
  polyhedral-SDP relaxation\\[-3pt]
  \noindent\rule{\linewidth}{1pt}
   $D=\R^n_+\rightarrow $ DNN relaxation
  };
  \node[block, below=of sq,text width=5cm]       (mom2)  
  {
  polyhedral-SDP relaxation\\[-3pt]
  \noindent\rule{\linewidth}{1pt}
  $\phantom{_+}$$D=\R^n \rightarrow $ SDP relaxation
  };

  \draw[arrow] (prob)  -- (slack);
  \draw[arrow] (prob)  -- (sq);
  \draw[arrow] (slack) 
  -- node[midway, right]   {$\A_1$}
  (mom1);
  \draw[arrow] (sq)    -- node[midway, right]   {$\A_2$} (mom2);
  \draw[<->, thick] (mom1) -- node[midway, above]   {equivalent} (mom2);
\end{tikzpicture}
    \caption{Reformulation and relaxation of polynomial problems with inequality constraints.}
    \label{fig:reformulation}
\end{figure}
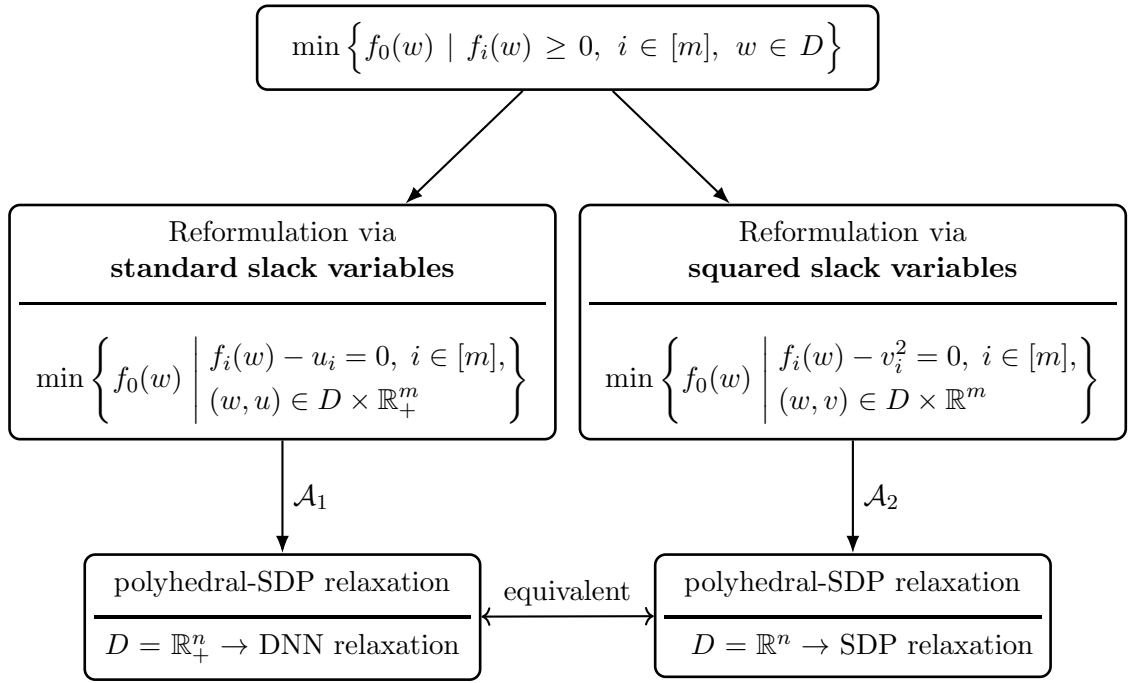
\end{center}

\subsection{Moment-SOS relaxations}\label{subsec-mom-sos-relax}
An alternative to the convex‐conic reformulation \eqref{eq-COP-convex} is provided by the moment–SOS hierarchy~\cite{lasserre2001global}, which approximates the general polynomial optimization problem
\begin{equation}\label{eq-eq-ieq-general}
    \min\Big\{f_0(w)\mid g_i(w)=0,\; i\in[\ell],\; h_j(w)\ge0,\; j\in[p]\Big\}
\end{equation}
via a sequence of semidefinite relaxations whose optimal values converge to the global optimum under mild assumptions as $\tau\to\infty$~\cite{lasserre2001global,tran2025convergence}.
Here we assume that $D$ can be expressed entirely by a finite collection of polynomial equalities and inequalities, and hence the condition $w\in D$ is absorbed into the families $\{g_i\}$ and $\{h_j\}$.
To set up the $\tau$th‐order moment-SOS relaxation, choose
$$
\tau\ge\max\left\{\left\lceil \frac{\deg f_0}{2}\right\rceil,\; \left\lceil \frac{\deg g_i}{2}\right\rceil\left( i\in[\ell]\right),\;
\degg{h_j}:=\left\lceil\frac{\deg h_j}{2}\right\rceil \left( j\in[p]\right)\right\}.
$$
Then consider the homogenized monomials: 
\[
x=\begin{bmatrix}
    x_0\\
    x_1\\
    \vdots\\
    x_n
\end{bmatrix}=
\begin{bmatrix}
    1\\
    w_1\\
    \vdots\\
    w_n
\end{bmatrix},
\quad v(w) := [w]_\tau \in \mathbb{R}^{\bar{n}_\tau},
\quad u(x) :=x_0^\tau \;v\left(\frac{w}{x_0}\right)\in \mathbb{R}^{\bar{n}_\tau}.
\]
We refer the readers to Subsection~\ref{subsec-notations} for the definition of $[w]_\tau$ and $\barn_\tau$. For example, when $\tau=2$ and $n=2$, we have
\[
\begin{array}{cll}
    x   &=[x_0,x_1,x_2]^\top=[1,w_1,w_2]^\top&\in\R^3,\\[3pt]
    v(w)&=[1,w_1,w_2,w_1^2,w_1w_2,w_2^2]^\top&\in \R^6,\\[3pt]
    u(x)&=[x_0^2,x_0x_1,x_0x_2,x_1^2,x_1x_2,x_2^2]^\top&\in \R^6.
\end{array}
\]
Each equality constraint $g_i(w)=0$ (with $\deg g_i \le2\tau$) generates a collection of redundant equalities. Concretely, for $i\in[\ell]$, one enforces
\begin{equation}\label{eq-fi}
    g_i(w) \cdot [w]_{2\tau - \deg g_i } = 0
\qLRq
f_{i,k}(w)=0,\quad k\in[m_i]
\end{equation}
so that there are a total of $m=\sum_{i=1}^\ell m_i$ redundant polynomials $f_{i,k}(w)$ with $m_i := \barn_{2\tau -\deg g_i}$, and $i\in[\ell]$.  Each of these can be homogenized as
\[
\bar f_{i,k}(x) = \langle Q^{i,k}, u(x)u(x)^\top \rangle
\qLRq
f_{i,k}(w) = \langle Q^{i,k}, v(w)v(w)^\top \rangle
\]
with $Q^{i,k} \in \mathbb{S}^{\barn_\tau}$.
In the moment relaxation, one replaces
\[
u(x)\,u(x)^\top\;\longmapsto\;X\;\in\;\mathbb{S}_+^{\barn_\tau},
\]
and imposes the linear constraints
\[
\langle Q^{\,i,k},\,X\rangle = 0, 
\quad \;k\in [m_i],\; i\in [\ell].
\]
Each inequality $h_j(w)\ge0$ (with $\deg(h_j)\le\tau$) corresponds to a localizing moment matrix:
\[
\M_{h_j}(v(w)v(w)^\top):=h_j(w)[w]_{\tau-\degg{h_j}}[w]^\top_{\tau-\degg{h_j}}\succeq 0
\;\;\longmapsto\;\;
\M_{h_j}(X)\succeq0.
\]
Combining both types of constraints with $\langle H^0,X\rangle=1$ (to represent $x_0^{2\tau}=1$), the $\tau$th‐order moment‐SOS relaxation of \eqref{eq-eq-ieq-general} takes the following form, after relabeling
$\{Q^{i,k}: k\in[m_i], i\in [l]\}$ as $\{Q^i: i\in[m]\}$: 
\begin{equation}\label{eq-eq-ieq-general-moment}
\min\left\{\<Q^0,X\> \;\middle|\;
\begin{aligned}
    &\langle H^0, X \rangle=1,\quad
    \langle Q^i, X \rangle=0,\;i\in [m],\\
    &\M_{h_j}(X)\succeq 0,\; j\in[p],\quad 
    X\in\bS^{\cA_\tau}_+\cap\L^{\cA_\tau}
\end{aligned}
\right\}.
\end{equation}
In this subsection, we consider three canonical cases:
\begin{itemize}
  \item[$\bullet$] Equality-constrained POP ($D=\R^n$), 
  \item[$\bullet$] Nonnegativity‐constrained POP ($D=\R^n_+$), or 
  \item[$\bullet$] General inequality-constrained POP ($h_j(w)\ge0$).
\end{itemize}
These cases cover many practical applications, and more general domains can be addressed using the reformulation techniques introduced in Subsection~\ref{subsec-ineq}. 
In each case, the moment–SOS relaxation \eqref{eq-eq-ieq-general-moment} is equivalent to the polyhedral‐SDP relaxation obtained via slack or squared‐slack reformulations (as illustrated in Figure~\ref{fig:moment}). 
Note that the polyhedral‐SDP relaxation is tighter than the moment–SOS shown in Figure \ref{fig:moment-n} for the nonnegativity case holds true only when $\tau = 1$. Moreover, the monomial basis $\cA$ used in Figure \ref{fig:moment-i} differs for \eqref{eq-POPI} and (\text{POP$_i^{s}$}), which corresponds to the reformulation via standard slack variables shown in Figure~\ref{fig:reformulation}.

\begin{figure}[ht]
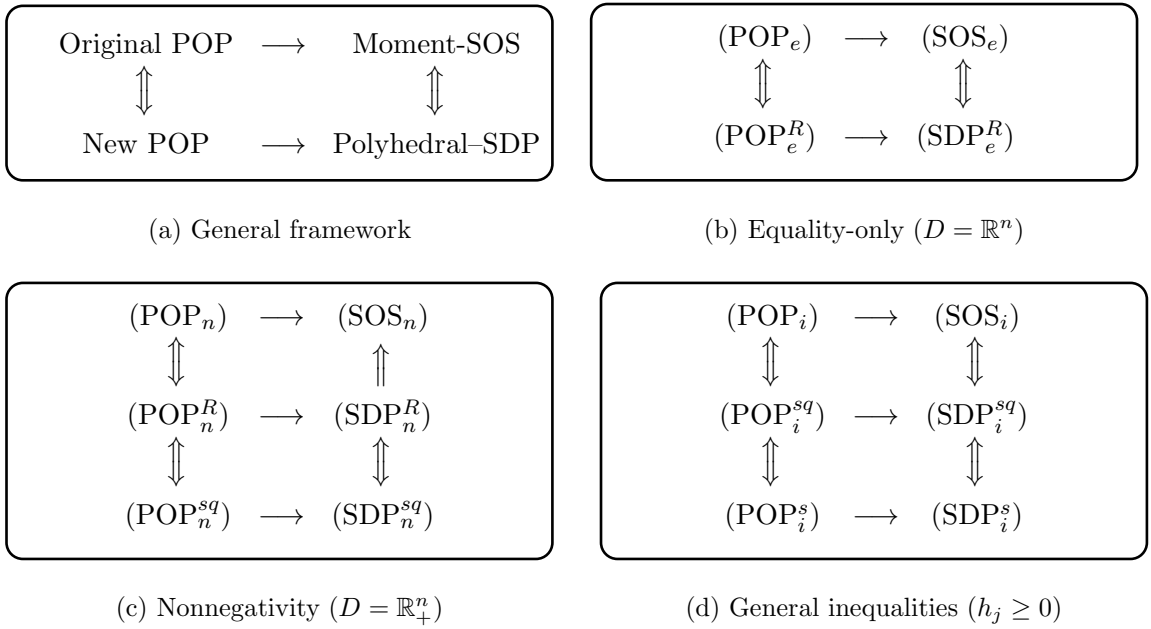

\begin{subfigure}[t]{0.48\textwidth}
    \centering
    \begin{tcolorbox}[colframe=black, colback=white, boxrule=1pt, arc=2mm, top=-7pt, bottom =3pt]
      \[
      \begin{array}{ccc}
    \text{Original POP} &\longrightarrow &  \text{Moment-SOS}\\
    \rotatebox{90}{$\Longleftrightarrow$} 
           & & 
         \rotatebox{90}{$\Longleftrightarrow$} \\
    \text{New POP} & \longrightarrow & \text{Polyhedral–SDP}
	\end{array}
      \]
    \end{tcolorbox}
    \subcaption{General framework}\label{fig:moment-general}
  \end{subfigure}
\hfill
  \begin{subfigure}[t]{0.48\textwidth}
    \centering
    \begin{tcolorbox}[colframe=black, colback=white, boxrule=1pt, arc=2mm, top=2pt]
      \[
      \begin{array}{ccc}
         \eqref{eq-POPE}  & \longrightarrow & \eqref{eq-SOSE} \\
         \rotatebox{90}{$\Longleftrightarrow$} 
           & & 
         \rotatebox{90}{$\Longleftrightarrow$} \\
         \eqref{eq-POPE-R} & \longrightarrow & \text{(SDP$_{e}^R$)}
      \end{array}
      \]
    \end{tcolorbox}
    \subcaption{Equality‐only ($D=\R^n$)}\label{fig:moment-e}
  \end{subfigure}
  \bigskip
  
   \begin{subfigure}[t]{0.48\textwidth}
    \centering
    \begin{tcolorbox}[colframe=black, colback=white, boxrule=1pt, arc=2mm, top=2pt]
      \[
      \begin{array}{ccc}
         \eqref{eq-POPN}  & \longrightarrow & \text{(SOS$_{n}$)} \\
         \rotatebox{90}{$\Longleftrightarrow$} 
           & & 
         \rotatebox{90}{$\Longrightarrow$} 
         \\
         \eqref{eq-POPN-R} & \longrightarrow & \eqref{eq-SDPN-R}\\
          \rotatebox{90}{$\Longleftrightarrow$} 
           & & 
         \rotatebox{90}{$\Longleftrightarrow$} \\
         \eqref{eq-POPN-SQ} & \longrightarrow & \text{(SDP$_{n}^{sq}$)}
      \end{array}
      \]
    \end{tcolorbox}
    \subcaption{Nonnegativity ($D=\R^n_+$)}\label{fig:moment-n}
  \end{subfigure}
  \hfill
  \begin{subfigure}[t]{0.48\textwidth}
    \centering
    \begin{tcolorbox}[colframe=black, colback=white, boxrule=1pt, arc=2mm, top=2pt]
      \[
      \begin{array}{ccc}
         \eqref{eq-POPI}  & \longrightarrow & \eqref{eq-SDPI} \\
         \rotatebox{90}{$\Longleftrightarrow$} 
           & & 
         \rotatebox{90}{$\Longleftrightarrow$} \\
         \eqref{eq-POPI-SQ} & \longrightarrow & \eqref{eq-SDP1-square}\\
         \rotatebox{90}{$\Longleftrightarrow$} 
           & & 
         \rotatebox{90}{$\Longleftrightarrow$} \\
         (\text{POP$_i^{s}$})   & \longrightarrow & (\text{SDP$_i^{s}$})
      \end{array}
      \]
    \end{tcolorbox}
    \subcaption{General inequalities ($h_j\ge0$)}\label{fig:moment-i}
  \end{subfigure}

  \caption{Equivalence between polyhedral-SDP relaxation and moment-SOS relaxation. Note that the polyhedral‐SDP relaxation \eqref{eq-SDPN-R} is tighter than the moment–SOS relaxation \text{(SOS$_{n}$)}
  in (c)  holds true only for the case $\tau=1$. We refer the reader to the following subsections for the meaning of various problems.}
  \label{fig:moment}
\end{figure}

\subsubsection{Equality-constrained problem}
Consider the polynomial optimization problem with only equality constraints:
\begin{equation}\label{eq-POPE}
    \min \Big\{ f_0(w) \mid  g_i(w) = 0, i\in [\ell],\; w\in\R^n\Big\}.\tag{POP$_{e}$}
\end{equation}
The $\tau$-order moment relaxation of the form
\eqref{eq-eq-ieq-general-moment} is given by
\begin{equation}\label{eq-SOSE}
\min\left\{\<Q^0,X\> \mid\ \langle H^0, X \rangle=1,\; \langle Q^i, X \rangle=0,\;i\in [m],\; X\in\bS^{\cA}_+\cap\L^\cA \right\}.
\tag{SOS$_{e}$}
\end{equation}
Consider the equivalent problem of \eqref{eq-POPE} with redundant constraints:
\begin{equation}\label{eq-POPE-R}
    \min \Big\{ f_0(w) \mid   f_i(w) = 0, i\in [ m],\; w\in\R^n\Big\},\tag{POP$_{e}^{R}$}
\end{equation}
where $f_i(w)=0$ collects all the equalities generated by \eqref{eq-fi}. Then one can verify that the polyhedral-SDP relaxation of the form
\eqref{eq-SDP} for \eqref{eq-POPE-R} coincides with \eqref{eq-SOSE}.

\subsubsection{Nonnegativity-constrained problem} 
Consider the problem with nonnegative variables:
\begin{equation}\label{eq-POPN}
    \min \Big\{ f_0(w) \mid  g_i(w) = 0, i\in [\ell],\; w\in\R^n_+\Big\}.\tag{POP$_{n}$}
\end{equation}
Similarly, the equivalent problem of \eqref{eq-POPN} with redundant constraints is given as:
\begin{equation}\label{eq-POPN-R}
    \min \Big\{ f_0(w) \mid   f_i(w) = 0, i\in [ m],\; w\in\R^n_+\Big\}.\tag{POP$_{n}^{R}$}
\end{equation}
Its polyhedral‐SDP relaxation  of the form \eqref{eq-SDP} is given by
\begin{equation}\label{eq-SDPN-R}
\min\left\{\<Q^0,X\> \mid\ \langle H^0, X \rangle=1,\; \langle Q^i, X \rangle=0,\; i\in [m],\; X\in\bS^{\cA}_+\cap \bN^\cA\cap\L^\cA  \right\}.\tag{SDP$_{n}^{R}$}
\end{equation}
We can also replace $w$ in \eqref{eq-POPN}
by the product of the squared-variable $s\in\R^n$, yielding
\begin{equation}\label{eq-POPN-SQ}
    \min \Big\{ f_0(s\circ s) \mid   f_i(s\circ s) = 0,\; i\in [ m],\; s\in\R^n\Big\}.\tag{POP$_{n}^{sq}$}
\end{equation}
By choosing \(\bar\cA = 2\,\cA_\tau\) and enforcing \(\P^{\bar\cA} = \bN^{\bar\cA}\), the polyhedral‐SDP relaxation 
of \eqref{eq-POPN-SQ}
coincides with \eqref{eq-SDPN-R}. In particular, when $\tau=1$, \eqref{eq-SDPN-R} is tighter than the first‐order moment–SOS relaxation of \eqref{eq-POPN} due to the additional nonnegativity constraints.

\subsubsection{Inequality-constrained problem} 
Consider the polynomial optimization problem with only inequality constraints:
\begin{equation}\label{eq-POPI}
    \min \Big\{ f_0(w) \mid  h_j(w)\geq 0,\; j\in[p],\; w\in\R^n\Big\}.\tag{POP$_{i}$}
\end{equation}
Its $\tau$th‐order moment–SOS relaxation is given by
\begin{equation}\label{eq-SDPI}
\zeta^{(1)}=\min\left\{\<Q^0,X\> \mid \langle H^0, X \rangle=1,\;  \M_{h_j}(X)\succeq 0,\, j\in[p],\;  X\in\bS^{\cA}_+\cap \L^\cA \right\}.\tag{SOS$_{i}$}
\end{equation}
Equivalently, we can introduce squared‐slack variables $s_j\in\mathbb{R}$ and rewrite \(\eqref{eq-POPI}\) as
\begin{equation}\label{eq-POPI-SQ}
\min \Big\{ f_0(w) \mid  h_j(w)-s_j^2= 0,\; j\in[m],\; (w;s)\in\R^{n+m}\Big\}.\tag{POP$_{i}^{sq}$}
\end{equation}
Let $\degg{h_j} = \lceil {\rm deg}(h_j)/2 \rceil$ for $j\in [p]$.
Choose the extended monomial basis
\begin{equation*}  
    v^{\bar \cA}(w):=\begin{bmatrix}
        [w]_\tau\\ 
        s_1 [w]_{\tau-\degg{h_1}}\\
        \vdots\\
        s_m [w]_{\tau-\degg{h_p}}\\
    \end{bmatrix}\in\R^{\bar \cA},\quad 
    \left|\bar \cA\right|=\barn_\tau +\underbrace{\sum_{j=1}^p 
    (\tau-\degg{h_j})}_{:=\bar p} 
\end{equation*}
and form the lifted moment matrix
\begin{equation*}
    v^{\bar \cA}(w)\left(v^{\bar \cA}(w)\right)^\top \in\bS^{\bar \cA}
    \quad \longmapsto \quad 
    Y:=\begin{pmatrix}
        Y^{11} & Y^{12}  \\
        Y^{21} & Y^{22} 
    \end{pmatrix}\in
    \begin{pmatrix}
        \bS^{\barn_\tau} & \R^{\barn_\tau\times {\bar p}}  \\
        \R^{{\bar p} \times \barn_\tau} & \bS^{{\bar p}} 
    \end{pmatrix}
    .
\end{equation*}
The corresponding polyhedral‐SDP relaxation of the form \eqref{eq-SDP} for
\eqref{eq-POPI-SQ} is
\begin{equation}\label{eq-SDP1-square}
\zeta^{(2)}=\min\left\{\<Q^0,Y^{11}\> \mid\ \langle H^0,Y^{11}\rangle=1,\, \H_j(Y)= 0,\, j\in[p],\;   Y\in\bS^{\bar \cA}_+ \cap \L^{\bar \cA}\right\},\tag{SDP$_{i}^{sq}$}
\end{equation}
where each linear map \(\H_j:\mathbb{S}^{\bar\cA}\to\mathbb{R}^{2(\tau-\degg{h_j})}\) encodes the identity
\begin{equation*}
\left(h_j(w)-s_j^2\right)
[w]_{2(\tau-\degg{h_j})} = 0
\quad\Longrightarrow\quad
\H_j(Y) = 0.
\end{equation*}
The following theorem establishes the equivalence of \eqref{eq-SDPI} and \eqref{eq-SDP1-square} and derives the reduced formulation.
\begin{theorem}
(1) The two relaxations \(\eqref{eq-SDPI}\) and \(\eqref{eq-SDP1-square}\) are equivalent, i.e., $\zeta^{(1)}=\zeta^{(2)}$.
\\
(2) The relaxation \(\eqref{eq-SDP1-square}\) is equivalent to the following reduced problem:
\begin{equation}\label{eq-SDP1-square-reduced}
\zeta^{(2)}=\min\left\{\<Q^0,Y^{11}\> \;\middle|\; 
\begin{array}{l}\langle H^0,Y^{11}\rangle=1,\; 
\M_{h_j}(Y^{11})-Y^{22}_j = 0,\; \forall\; j\in[p],
\\[3pt]
Y^{11} \in\bS^{ \cA}_+ \cap \L^{\cA},\;
Y^{22}_j \in \bS^{\cA_j}_+\cap \L^{\cA_j}, \;\forall\; j\in[p]
\end{array}
\right\},
\end{equation}
where $\cA_j$ is the index set of exponents of $s_j[w]_{\tau - \degg{h_j}}$.
Note that \eqref{eq-SDP1-square-reduced} is exactly 
the moment-SOS problem in \eqref{eq-SDPI} by introducing the auxiliary matrices
$Y^{22}_j$, $j\in[p]$.
\end{theorem}
\begin{proof}
(1) First we prove that $\zeta^{(1)}\leq \zeta^{(2)}$. Suppose Y is feasible for \(\eqref{eq-SDP1-square}\). Then $\H_j (Y)=0$ implies 
\[
h_j(w)[w]_{\tau-\degg{h_j}}[w]_{\tau-\degg{h_j}}^\top=
s_j^2[w]_{\tau-\degg{h_j}}[w]_{\tau-\degg{h_j}}^\top
\quad \Longrightarrow \quad
\M_{h_j}(Y^{11}) = 
Y^{22}_{j}\succeq 0 ,
\]
where $Y^{22}_j\in\bS^{\tau-\degg{h_j}}_+$ denotes the $(j,j)$-th diagonal block of $Y^{22}$. Thus, $X=Y^{11}$ is feasible for \(\eqref{eq-SDPI}\) with the same objective value, yielding $\zeta^{(1)}\le\zeta^{(2)}$.\\
Next we prove that $\zeta^{(1)}\geq \zeta^{(2)}$. Conversely, let $X$ be feasible for \(\eqref{eq-SDPI}\).  Define
\begin{equation*}
Y:=\begin{pmatrix}
X&0_{\barn_\tau\times {\bar p}}\\0_{{\bar p}\times \barn_\tau }&\M_h(X)
\end{pmatrix},\quad 
\M_h(X):=\operatorname{blkdiag}(\M_{h_1}(X), \dots, \M_{h_p}(X)).
\end{equation*}
Then $Y$ is a feasible for \eqref{eq-SDP1-square} with the same objective function, so $\zeta^{(1)}\geq \zeta^{(2)}$.
\\[5pt]
(2) 
One may observe that the variable $Y^{12}$ in \eqref{eq-SDP1-square}, which correspond
to the monomials in $s_j [w]_\tau [w]_{\tau -\degg{h_j}}^\top$ for $j\in [p]$, is not
involved in the data matrices in \eqref{eq-SDP1-square}. Thus, we may set
$Y^{12}$ and $Y^{21}=(Y^{12})^\top$ to be zero without affecting the optimal value of  \eqref{eq-SDP1-square}. Similarly, for the $(j,k)$-block of $Y^{22}$,
which corresponds to the monomials in 
$s_js_k[w]_{\tau -\degg{h_j}}[w]_{\tau -\degg{h_k}}^\top$, is not involved in the data matrices in \eqref{eq-SDP1-square}
for $1\leq j<k\leq p$, and we can set these blocks to zero without 
affecting the optimal value. Thus we may restrict $Y^{22}$ to be a block 
diagonal matrix with the $(j,j)$-th diagonal block $Y^{22}_j$ corresponding to the monomials
$(s_j[w]_{\tau -\degg{h_j}})(s_j[w]_{\tau -\degg{h_j}})^\top$ for $j\in [p].$
To summarize, instead of \eqref{eq-SDP1-square}, we can equivalently consider
its reduced version: 
\begin{equation*} 
\zeta^{(2)}=\min\left\{\<Q^0,Y^{11}\> \;\middle|\; 
\begin{array}{l}\langle H^0,Y^{11}\rangle=1,\; \widehat{\H}_j(Y^{11},Y^{22}_j)= 0,\, j\in[p],
\\[3pt]
Y^{11} \in\bS^{ \cA}_+ \cap \L^{\cA},\;
Y^{22}_j \in \bS^{\cA_j}_+\cap \L^{\cA_j}, \; j\in[p]
\end{array}
\right\},
\end{equation*}
where  $\widehat{\H}_j(Y^{11},Y^{22}_j)=0$ encodes the equality constraint
$(h_j(w)-s_j^2)[w]_{2(\tau - \degg{h_j})}=0$ in terms of the new variables 
$Y^{11}$ and $Y^{22}_j$. We can readily show that 
$\widehat{\H}_j(Y^{11},Y^{22}_j)=0$ is equivalent to the following condition:
$$ 
\M_{h_j}(Y^{11}) - Y^{22}_j = 0\quad \forall\; j\in [p].
$$
Thus the problem \eqref{eq-SDP1-square-reduced} is equivalent to
\eqref{eq-SDP1-square}.
\end{proof}

\subsection{Facial structure}\label{subsec-FR}
As mentioned in the introduction, in many polynomial optimization relaxations of the form \eqref{eq-SDP}, one encounters equality constraints of the form
\begin{equation}\label{eq-FR}
    AX = 0,\;X\succeq 0
\quad\Longleftrightarrow\quad
\langle A^\top A, X \rangle = 0,\;X\succeq 0
\quad\Longleftrightarrow\quad
X \in (A^\top A)^\perp \cap \bS^\cA_+.
\end{equation}
That is, the feasible set is restricted to a face of the semidefinite cone $\bS^\cA_+$ exposed by the PSD matrix $A^\top A\in\bS^{\cA}_+$.
(Similarly, if \eqref{eq-SDP} has constraints $\inprod{Q_i}{X}=0$ 
with $Q_i\succeq 0$
for some $i\in[l]$, we can also replace them
by $B_iX = 0$ for some full row-rank matrix $B_i$ satisfying 
$B_i^\top B_i = Q_i$.)
Such “facial constraints” arise naturally in both reformulation-linearization techniques (RLT) \cite{sherali2007rlt,qiu2024polyhedral,RiNNAL+} and moment-SOS hierarchies \cite{lasserre2001global,Lasserre2001}. Although facial constraints strengthen the relaxation, they also destroy Slater’s condition, which can make the dual unattainable, and often induce numerical ill--conditioning in SDP solvers \cite{drusvyatskiy2017many,RiNNAL}.

In this subsection, we first analyze the geometric structure of the facial constraints, and then illustrate how these facial constraints naturally arise in (i) RLT relaxations, (ii) moment–SOS relaxations, and (iii) polynomial‐optimization problems with binary variables. In Section~\ref{sec-alg}, we explain how our algorithm exploits this facial structure to avoid degeneracy.

\medskip

\noindent\tb{Geometric structure.} 
Let \(A\in\mathbb{R}^{\,m\times|\cA|}\) be a full-row-rank matrix.  Choose any full‐row–rank matrix \(U\in\mathbb{R}^{(|\cA|-m)\times |\cA|}\) whose rows form a basis of $\ker(A)$, so that $A U^\top=0$.  Then for every $X\succeq0$,
\begin{equation}\label{eq-AX-reduce}
    AX = 0
\quad\Longleftrightarrow\quad
\langle A^\top A, X \rangle = 0
\quad\Longleftrightarrow\quad
\exists\, \Theta \in \mathbb{S}^{|\cA|-m} :\; X = U^\top \Theta U.
\end{equation}
Hence, the feasible region \(\S^\cA_+\cap \mathcal{F} := \{ X \in\S^\cA \mid AX = 0 \} \) is a face of the PSD cone \( \mathbb{S}^{\cA}_+ \), consisting of all PSD matrices whose range is contained in \( \ker(A) \).
Let \( X \in \mathcal{F}\cap\S^\cA_+ \) be any feasible point. By~\eqref{eq-AX-reduce}, there exists \( \Theta \in \mathbb{S}^{|\cA| - m}_+ \) such that \( X = U^\top \Theta U \). The tangent cone to \( \mathcal{F}\cap\S^n_+ \) at \( X \) is then given by
\[
T_{\mathcal{F}\cap\S^\cA_+}(X) = \Big\{ H \in \mathbb{S}^{\cA} \,\mid\, H = U^\top \widehat{H} U,\; \widehat{H} \in T_{\mathbb{S}_+^{|\cA| - m}}(\Theta) \Big\}.
\]
This follows from the fact that \( \mathcal{F}\cap\S^\cA_+ \) is the linear image of the lower-dimensional PSD cone \( \mathbb{S}_+^{|\cA| - m} \) under the injective linear map \( \widehat{H} \mapsto U^\top \widehat{H} U \). Hence, the tangent directions at \( X \) correspond exactly to the lifted tangent directions at \( \Theta \).
The normal cone to \( \mathcal{F}\cap\S^\cA_+ \) at \( X \) is then given by
\begin{equation}\label{eq-normal-cone}
\begin{aligned}
    N_{\mathcal{F}\cap\S^\cA_+}(X) 
    &= \Big\{ S \in \mathbb{S}^{\cA} \,\mid\, U S U^\top \in N_{\mathbb{S}_+^{|\cA| - m}}(\Theta) \Big\}\\
    &= \Big\{ S \in \mathbb{S}^{\cA} \,\mid\, -U S U^\top \succeq 0,\; (U S U^\top) \Theta = 0 \Big\}\\
    &= \Big\{ S \in \mathbb{S}^{\cA} \,\mid\, -U S U^\top \succeq 0,\; USX = 0 \Big\}.
\end{aligned}
\end{equation}


\noindent\tb{Reformulation linearization technique (RLT).} 
The convex conic reformulation \eqref{eq-COP-convex} is equivalent to the original polynomial optimization problem \eqref{eq-POP} under suitable conditions. However, its polyhedral-SDP relaxation \eqref{eq-SDP} may remain weak when applied directly to \eqref{eq-POP}. To tighten this relaxation, we apply RLT, which multiplies each polynomial constraint by selected monomials to produce new equality constraints.  Although these additional equalities are redundant in the original variable space, they become nonredundant in the lifted SDP, further restricting the feasible set and yielding a tighter relaxation.

Although RLT can be applied to both equality and inequality constraints, we focus here on its use for equalities. The treatment of inequalities via slack or squared-slack reformulations follows from Subsection~\ref{subsec-ineq}.
Assume that \eqref{eq-POP} includes equalities $f_i(w)=0$ whose homogenized forms $\bar f_i(x)$ are linear in $u^\cA(x)$, i.e.,
\[
\bar f_i(x)= \langle a_i, u^\cA (x)\rangle \text{ for some } a_i\in\R^\cA.
\]
Collecting these coefficient vectors into $A:=[a_1,\dots,a_m]^\top\in\R^{m\times \cA}$. We see that $\bar f_i(x)=0$ is equivalent to \(A\,u^\cA(x)=0\).   Multiplying each equation by the vector $u^\cA(x)^\top$ yields the RLT-type constraints $AX=0$, which can be added to \eqref{eq-SDP} to tighten the relaxation. The derivation is summarized in Figure~\ref{fig:RLT-constraints}.
\begin{figure}[ht!]
\begin{tcolorbox}[colframe=black, colback=white, boxrule=1pt, arc=2mm, top=0pt]
\vspace{1em}
\[
\begin{array}{rcrcl}
\begin{bmatrix}
    \bar f_1(x) \\
    \vdots \\
    \bar f_m(x)
\end{bmatrix} = 0
&
\qLRq
&
\begin{bmatrix}
    \bar f_1(x) \\
    \vdots \\
    \bar f_m(x)
\end{bmatrix}\left(u^\cA(x)\right)^\top = 0
&&
\\[2em]
\multicolumn{1}{c}{\rotatebox{90}{$\Longleftrightarrow$} }
&&
\multicolumn{1}{c}{\rotatebox{90}{$\Longleftrightarrow$} }
&&
\\[.5em]
A \; u^\cA(x) = 0
&
\qLRq
&
A\,u^\cA(x)\,\left(u^\cA(x)\right)^\top = 0
&\qRq &
AX=0.
\end{array}
\]
\end{tcolorbox}
\caption{Derivation of RLT constraints with facial structure.}
\label{fig:RLT-constraints}
\end{figure}

\medskip
\noindent\tb{Moment-SOS relaxations.}
When applying the moment–SOS hierarchy to the general POP \eqref{eq-eq-ieq-general}, each equality $g_i(w)=0$ with $\deg(g_i)\le\tau$ gives rise to redundant polynomial constraints of the form
\begin{equation*}
    g_i(w)[w]_{2\tau-\deg (g_i)}=0,
\end{equation*}
which can be rewritten as
$$
\left(g_i(w)[w]_{\tau-\deg (g_i)}\right)[w]_\tau^\top=0
    \, \Longleftrightarrow \,
    G_i[w]_\tau[w]_\tau^\top=0
    $$
where \(G_i\in\mathbb{R}^{d_i\times\barn_\tau}\) is the coefficient matrix of the polynomials $g_i(w)[w]_{\tau-\deg (g_i)}$ with respect to the basis 
$[w]_\tau$, and
$$
d_i \;=\; \binom{\,n + \tau - \deg(g_i)\,}{\,\tau - \deg(g_i)\,}.
$$
Lifting $[w]_\tau[w]_\tau^\top$ to the matrix variable $X\in\bS^{\barn_\tau}_+$ yields the linear constraint
$$
G_iX =0,\quad i=1,\dots,\ell.
$$
Collecting the above constraint matrices into
$$
A =\begin{bmatrix}G_1\\ \vdots\\ G_\ell\end{bmatrix}
\;\implies\; AX = 0,
$$
we see that all equality constraints with $\deg(g_i)\le\tau$ in the moment–SOS relaxation can be written concisely as $AX=0$.

We illustrate the procedure in Figure~\ref{fig:RLT-constraints} by an example below.

\begin{example}
    Consider the nonnegative quadratic programming problem:
    \[
    \min\Big\{w^\top Q w + 2c^\top w\mid a_i^\top w=b_i,\; i\in[m],\; w\in\R^n_+\Big\}.
    \]
    Set $\tau=1$ and choose the full degree‐$1$ monomial basis \(\cA_1\).  Then
\[
u^\cA(x) \;=\;\bigl[x_0,x_1,\dots,x_n\bigr]^\top
\;=\;\bigl[\,1,\,w_1,\dots,w_n\,\bigr]^\top\in\mathbb{R}^{\,n+1}.
\]
By RLT, each linear constraint $a_i^\top w = b_i$ implies
\[
\begin{bmatrix}-\,b_i & a_i^\top\end{bmatrix}\,u^\cA(x)
\;=\; 0
\;\;\Longrightarrow\;\;
\begin{bmatrix}-\,b_i & a_i^\top\end{bmatrix}\,X \;=\; 0.
\]
By collecting all $m$ such rows, we define
$$
A \;=\;
\begin{bmatrix}
-\,b_1 & a_1^\top \\[4pt]
\vdots \\[4pt]
-\,b_m & a_m^\top
\end{bmatrix}
\;\in\;\mathbb{R}^{\,m\times(n+1)},
$$
so that $AX=0$ enforces all linear equalities in the lifted SDP.  The resulting polyhedral‐SDP (or moment–SOS) relaxation is
    $$
    \min 
    \Bigl\{\langle Q^0, X\rangle \;\Big|\;
    \langle H^0, X\rangle = 1,\;
    A X=0,\;
    X\in \mathbb{S}_+^{n+1} \cap \bN^{n+1}
    \Bigr\},
    $$
    where
    \[
    Q^0=\begin{bmatrix}
        0&c^\top\\c & Q
    \end{bmatrix},\quad 
    H^0=\begin{bmatrix}
        1 & 0_{n}^\top\\
        0_n & 0_{n\times n}
    \end{bmatrix}.
    \]
\end{example}

\medskip

A common approach to eliminate the affine constraint $AX=0$ in the SDP problem is facial reduction \cite{borwein1981facial}, which replaces $X$ with a lower-dimensional variable $\Theta$ as defined in \eqref{eq-AX-reduce}. This method has been applied to specific cases \cite[Section 3.2]{Lasserre2001}, where the constraints in \eqref{eq-POP} are binary, such as $x_i^2=x_i$ or $x_i^2=1$, and facial reduction reduces to removing redundant rows and columns. However, for general problems \eqref{eq-POP}, this process typically destroys the sparsity of the remaining linear constraints. In the next section, we present a more efficient approach that preserves sparsity while handling affine constraints.

\section{Algorithm}\label{sec-alg}

In this section, we introduce the proposed method RiNNAL-POP for solving the following general polyhedral-SDP problem:
\begin{equation}\label{eq-SDP-general}
\min\Big\{\<C,X\> \mid \Q(X)=b,\; X\in\F\cap \P \cap\bS^{n}_+\Big\},\tag{SDP}
\end{equation}
where $\Q: \bS^n \to \mathbb{R}^d$ is a given linear mapping with $b \in \mathbb{R}^d$, and $\P$ is a polyhedral convex set in $\bS^n$.
We assume that $\F$ has the compact form:
\[
  \mathcal F = \big\{X\in\bS^n \mid AX=0\big\},
\]
where $A \in \mathbb{R}^{m \times n}$ has full row-rank. \eqref{eq-SDP-general} is a generalization of the polyhedral-SDP relaxation~\eqref{eq-SDP}. In particular, one can take $\P$ as the intersection $\P^\cA\cap \L^\cA\cap \{X\in\S^\cA\mid\langle H^0,X\rangle =1\}$, i.e., the nonnegativity cone, the consistency cone, and the normalization hyperplane. The basic framework 
of RiNNAL-POP follows those in \cite{RiNNAL,RiNNAL+} that were originally developed to solve SDP-RLT and DNN relaxations of mixed-binary quadratic programs. 
Here we extend the algorithms in  \cite{RiNNAL,RiNNAL+} to address the problem~\eqref{eq-SDP-general}.
To apply the augmented Lagrangian framework, we rewrite \eqref{eq-SDP-general} in an equivalent splitting form
\begin{equation}\label{prob-ALM-reform}
    \min\left\{\<C,X\> + \delta_{\F\cap \S^n_+} (X)+\delta_{\P}(Y)  \mid\ X-Y=0,\; \Q (X)=b \right\},
\end{equation}
where $\delta_C$ denotes the indicator function of a convex set $C$. Once we have computed $(X^*,Y^*)$ (and the associated multipliers), we certify its global optimality by checking the following Karush–Kuhn–Tucker (KKT) conditions of \eqref{prob-ALM-reform}:
\begin{equation}\label{KKT}
    \begin{aligned}
        &X-Y=0,\;
        AX=0,\; 
        \Q(X)=b,\;  
        C-W-A^\top U-U^\top A-\Q^*(y) =S,\\
        &
        \langle X,{S}\rangle=0,\; 
        X \succeq 0,\; 
        S\succeq 0,\;
        W\in N_{\P}(Y),\; 
        Y\in\P,
    \end{aligned}
\end{equation}
where $U \in \mathbb{R}^{m \times n}$, $W\in\mathbb{R}^{n\times n}$, $y \in \mathbb{R}^d$, and $S \in \mathbb{S}^n$ are the dual variables. Here $N_\P (Y)$ denotes the normal cone of the set $\P$ at the point $Y$. We make the following assumption throughout the paper.
\begin{assumption}\label{assumption-mild}
    Problem \eqref{eq-SDP-general} admits an optimal solution satisfying the KKT conditions \eqref{KKT}, and its objective function is bounded from below.
\end{assumption}

\begin{remark}
  Although \eqref{eq-FR} shows that enforcing three different constraints yields the same primal feasible set, their dual formulations differ. In particular, when $A\neq0$, expressing the constraint as $\langle Q^i,X\rangle=0$ or $\langle A A^\top,X\rangle=0$ exacerbates the failure of Slater’s condition and can enlarge the duality gap.  In contrast, the linear form $AX=0$ produces the smallest duality gap among these alternatives—and under mild regularity conditions attains zero duality gap (see \cite{bomze2017fresh,RiNNAL,RiNNAL+}).  In Section~\ref{subsec-recover-dual}, we explain how to recover the corresponding dual multipliers and how to handle the enlarged constraint system efficiently in our algorithm.
\end{remark}

\begin{remark}
    In \eqref{eq-SDP-general}, we enforce the constraint $\langle H^0, X \rangle = 1$ in the set $\mathcal{P}$ rather than the feasible set $\mathcal{F}$ for simplicity of presentation. However, for polyhedral-SDP relaxations of polynomial optimization problems involving mixed-binary constraints (i.e., $x_i \in \{0,1\}$), it is preferable to include this constraint in $\mathcal{F}$. This treatment is adopted in \cite{RiNNAL, RiNNAL+} and typically leads to a performance improvement of at least 50\%.
\end{remark}

\subsection{Augmented Lagrangian method}
RiNNAL-POP is an augmented Lagrangian method for solving \eqref{eq-SDP-general}.
The augmented Lagrange function of its equivalent split form \eqref{prob-ALM-reform} is defined by
\begin{align*}
   & L_\sigma(X,Y;y,W) \;:=\; \<C,X\>  - \<y,\Q (X)-b\> -\langle W,X-Y\rangle + \frac{\sigma}{2}\left\|\Q (X)-b\right\|^2 + \frac{\sigma}{2}\|X-Y\|^2\\
    &=\<C,X\>  + \frac{\sigma}{2}\left\|\Q (X)-b-\sigma^{-1}y\right\|^2+\frac{\sigma}{2}\|X-Y-\sigma^{-1}W\|^2-\frac{1}{2\sigma}\|y\|^2-\frac{1}{2\sigma}\|W\|^2.
\end{align*}
Given the initial penality parameter $\sigma_0>0$, dual variable $y^0\in \R^{d}$ and $W^0\in\R^{n\times n}$, the augmented Lagrangian method performs the following steps at the $(k+1)$-th iteration:
\begin{align}
    (X^{k+1},Y^{k+1})&=\arg\min\ \Big\{L_{\sigma_k}(X,Y;y^k,W^k) :\ X\in {\F\cap \S^n_+},\; Y\in\P \Big\},\label{ALM-sub-XY}\\
    y^{k+1}&=y^k-\sigma_k(\Q (X^{k+1})-b),\notag\\
    W^{k+1}&=W^k-\sigma_k(X^{k+1}-Y^{k+1}),\notag
\end{align}
where $\sigma_k \uparrow \sigma_{\infty} \leq+\infty$ are positive penalty parameters. For a comprehensive discussion on the augmented Lagrangian method applied to convex optimization problems and beyond, see \cite{hestenes1969multiplier, powell1969method, rockafellar1976augmented}. 
Let $\widetilde y$ and $\widetilde W\in \bS^{n}$ be fixed. The inner problem \eqref{ALM-sub-XY} can be expressed as:
\begin{equation}
    \min\ \Big\{L_{\sigma}(X,Y;\widetilde y,\widetilde W) :\ X\in {\F\cap \S^n_+},\; Y\in\P \Big\}.\label{ALM-sub-XY-fix}
\end{equation}
In \eqref{ALM-sub-XY-fix}, we can first minimize with respective to $Y\in\P$ to get the following convex optimization problem related only to $X$:
\begin{equation}\label{ALM-sub-Y}
    \min\left\{
    \begin{aligned}
        \phi (X) &:= \langle C,X\rangle 
        + \frac{\sigma}{2} \|\sigma^{-1}y-(\Q(X)-b)\|^2 \\
        &+ \frac{\sigma}{2}\|X-\sigma^{-1}W-\Pi_{\P}(X-\sigma^{-1}W)\|^2
    \end{aligned}
    \;\middle|\; X\in {\F\cap \S^n_+} \right\},\tag{CVX}
\end{equation}
Once we obtain the optimal solution $\wt X$ of \eqref{ALM-sub-Y}, we can recover the optimal solution $\wt Y=\Pi_P(\wt X-\sigma^{-1}\wt W)$.
The ALM framework for solving \eqref{prob-ALM-reform} is summarized in Algorithm \ref{alg1}, where $X^{k+1}$ is obtained by the factorization described in the next subsection. 
\begin{algorithm}
\linespread{1.1}\selectfont
\caption{The RiNNAL-POP method}
\label{alg1}
\begin{algorithmic}[1]
\STATE {\bf Parameters:} Given $\sigma_0>0$, integer $r_0>0$, initial point $R^0\in\R^{n\times r_0}$.
\STATE $k\gets 0$, $i\gets 0$, $y^0=0_{d}$, $W^0=0_{n\times n}$.
\WHILE{\eqref{eq-SDP-general} is not solved to required accuracy}
    \WHILE{\eqref{ALM-sub-Y} is not solved to required accuracy}
        \STATE Obtain $R^{i+1}$ by solving \eqref{ALM-sub-LR-M} satisfying \eqref{LR-stop}.
        \STATE Obtain $\widehat X^{i+1}=R^{i+1}(R^{i+1})^\top$.
        \STATE Obtain $X^{i+1}$ by solving \eqref{ALM-sub-Y} via one step of PG satisfying \eqref{CVX-stop}.
        \STATE Obtain $R^{i+1}$ from the eigenvalue decomposition  of $X^{i+1}$.
    \ENDWHILE
    \STATE $X^{k+1}=X^{i}.$
    \STATE $Y^{k+1}=\Pi_P(X^{k+1}-\sigma^{-1}W^k)$.
    \STATE $y^{k+1}=y^k-\sigma_k(\Q (X^{k+1})-b)$.
    \STATE $W^{k+1}=W^k-\sigma_k(X^{k+1}-Y^{k+1})$.
    \STATE {Update $\sigma_k$}.
    \STATE $k\gets k+1$, \;$i\gets 0$.
\ENDWHILE
\end{algorithmic}
\end{algorithm}

\subsection{Low-Rank Phase}\label{subsec-LR-phase}

In this phase, we follow the approach introduced in \cite{RiNNAL,RiNNAL+} to solve \eqref{ALM-sub-Y}.  Let $\sigma>0$, $y\in\R^d$, and $W\in\R^{n\times n}$ be fixed, and suppose that the subproblem \eqref{ALM-sub-Y} admits an optimal solution of rank $r>0$.  By the Burer–Monteiro (BM) factorization, any rank-$r$ solution can be written as  
\[
X\;\in\;\mathcal{F}\,\cap\,\S_+^n\qRq X \;=\; R R^\top
 \text{ for some } R\in\R^{n\times r}.
\]
The crucial property
\begin{equation} \label{eq-reduction}
AX = 0 
\;\Longleftrightarrow\; 
A (R R^\top) = 0 
\;\Longleftrightarrow\; 
A R = 0
\end{equation}
shows that enforcing $X\in\F$ (i.e.\ $AX=0$) is equivalent to the {affine} constraint $AR=0$ on $R$.  Consequently, the original SDP subproblem \eqref{ALM-sub-Y} is equivalent to the following factorized (nonconvex) problem:
\begin{equation}\label{ALM-sub-LR-M}
\min\Bigl\{\,f_r(R) := \phi(RR^\top)\mid A R = 0 \Bigr\},
\tag{LR}
\end{equation}
where the constraint $AR=0$ defines an affine subspace of dimension $(n-m)r$.  In this form, 
the decision‐variable's dimension is reduced from $n\times n$ to $n\times r$, and the number of linear constraints is effectively reduced from the original $mn$ equations on the $n\times n$ matrix variable $X$ to $mr$ equations on $R$.
Thus, one can directly apply any suitable first‐order method (for example, projected gradient descent) to \eqref{ALM-sub-LR-M} without explicitly handling the full $n\times n$ PSD cone. In practice, this reduction in both the dimension and number of constraints often yields substantial savings in memory and per‐iteration cost. 
The reduction enabled by the crucial property \eqref{eq-reduction} 
is key to the success of our final algorithm RiNNAL-POP.

We impose only the following mild, implementable non-increasing objective condition on the low-rank phase:
\begin{equation}\label{LR-stop}
    f_r(R^{i+1})\leq f_r(R^i)
    \quad\Leftrightarrow\quad
    \phi_r(R^{i+1}(R^{i+1})^\top)\leq 
    \phi_r(R^{i}(R^{i})^\top).
\end{equation}

The subproblem \eqref{ALM-sub-LR-M} is nonconvex, so the algorithm must be capable of escaping from saddle points. Moreover, the factorized formulation is equivalent to the original SDP only when the chosen rank exceeds the true optimal rank. To enforce these requirements, we introduce a convex lifting phase that automatically increases $r$ as needed and guarantees a monotonic decrease in the objective value to escape saddle points, thereby ensuring convergence to a global solution.

\subsection{Convex lifting phase}\label{subsec-CVX-phase}
In the convex lifting phase, we adopt the strategy proposed in \cite{lee2022escaping} to ensure global convergence of RiNNAL+ while automatically adjusting the rank.  Let $\sigma>0$, $y\in \R^{d}$, and $W\in\R^{n\times n}$ be fixed.  Given any feasible starting point $\widehat{X}\in \F\cap\S^n_+$, we conduct a single projected‐gradient step on the SDP‐subproblem \eqref{ALM-sub-Y}:
\begin{align}\label{eq-CVX-PG}
    X=\Pi_{\F\cap \mathbb{S}_{+}^{n}}\left(\widehat{X}-t\nabla \phi(\widehat{X})\right)
    =\arg\min\Big\{ \|X-G\|^2  \mid\ X \in \F\cap \mathbb{S}_{+}^{n} \Big\},
\end{align}
where $G:=\widehat{X}-t\nabla \phi(\widehat{X})$ and $t>0$ is an appropriate stepsize. One can compute $X$ efficiently in closed form (cf.\ \cite[Lemma 2]{RiNNAL+}):
\begin{equation*}
    X=\Pi_{\bS^n_+} (JGJ),\quad J:=I-A^\top (AA^\top)^{-1}A.
\end{equation*}

For the stepsize choices, the only requirement of our framework is that the final stepsize of each ALM subproblem is uniformly bounded and satisfies the following criteria for some given $\delta\in (0,1)$:
\begin{equation}\label{CVX-stop}
    \phi (X^{i+1})\leq \phi (X^i)+\langle \nabla\phi(\widehat X^{i}),X^{i+1}-\widehat{X}^i \rangle+\frac{\delta}{t_i}\|X^{i+1}-\widehat{X}^i\|^2.
\end{equation}
To ensure that $t_i$ is bounded, we need to specify an upper $t_{\text{max}}$ and a lower bound
$t_{\text{min}}$. It is known that if $\nabla \phi$ is $L$-Lipschitz continuous, then \eqref{CVX-stop} holds for any $t_i\leq 2\delta/L$,
and thus we assign $t_{\text{min}}\leq 2\delta/L$ to ensure that our algorithm is well-defined. Since the convex lifting phase is not the major tool for reducing the objective value, we simply assign a fixed step size $t_i=t$ in our implementation.

This phase offers several advantages:
\begin{itemize}
  \item {Primal feasibility enforcement:} a single projection restores primal feasibility of \(X\).
  \item {Convergence guarantee:} the projection helps to escape from saddle points in the low-rank phase and yields a monotonic decrease in the objective with a suitably chosen stepsize, thus ensuring global convergence of the ALM subproblem. For a more detailed analysis, see~\cite{lee2022escaping}.
  \item {Automatic rank selection:} the lifting phase automatically identifies a rank \(r\) satisfying \(r \ge r^*\) to ensure equivalence between the low-rank subproblem \eqref{eq-ALMsub-LR} and the SDP-subproblem \eqref{ALM-sub-Y}, while keeping \(r\) as small as possible to minimize memory usage and per-iteration cost. We provide a rigorous theoretical justification for this behavior from the perspective of rank identification in the next subsection.
\end{itemize}
Together, these properties enable a reliable initialization, provide an accurate rank estimate, and enhance computational efficiency, ultimately ensuring the global convergence of RiNNAL-POP.

\subsection{Rank identification}\label{subsec-rank-identification}
In the convex lifting phase, we use a projected‐gradient scheme to solve the augmented Lagrangian subproblem~\eqref{ALM-sub-Y}, which can be equivalently written as
\begin{equation}\label{eq-CVX-short}\tag{CVX}
    \min_X\; F(X)\coloneqq\phi(X)+\Psi(X),
\end{equation}
where the function $\phi$ is lower-bounded, convex, and diﬀerentiable with $L$-Lipschitz continuous gradient, and the function $\Psi$ is convex and has the form
\begin{equation}\label{eq:Psi}
    \Psi(X)\coloneqq \delta_{\F\cap\S^n_+}(X), \quad \mathcal F = \big\{X\in\bS^n \mid AX=0\big\}.
\end{equation}
Empirically, these projected-gradient iterates very rapidly “lock on’’ to the true rank of the optimal solution of~\eqref{eq-CVX-short}, see~\cite{RiNNAL+} for numerical comparisons highlighting the advantage of the PG step over adaptive rank-tuning strategies.
To explain this behavior rigorously, we establish a finite‐step rank identification result: under a standard nondegeneracy condition, the PG iterates lie on the smooth manifold of fixed‐rank matrices after finitely many steps. This provides the theoretical justification for the practical efficiency of our convex lifting phase in identifying the true rank of the optimal solution of \eqref{eq-CVX-short}. 

Before proceeding to the convergence analysis, we introduce the definition of convex partly smooth functions.  This concept is based on the notion of a \(C^2\)–manifold, that is, a manifold defined by equations that are twice continuously differentiable.

\begin{definition}[\tb{Partly smooth \cite{lewis2002active}}]
\label{def:ps}
A convex function $\Psi$ is partly smooth at a point $X^*$ relative
to a set $\M$ containing $X^*$ if $\partial \Psi(X^*) \neq \emptyset$
and:
\begin{enumerate}
		\itemsep 0pt
		\topsep 0pt
	\item Around $X^*$, $\M$ is a $\mathcal{C}^2$-manifold and
		$\Psi|_{\M}$ is $\mathcal{C}^2$.
	\item The affine span of $\partial \Psi(X)$ is a translate
		of the normal space to $\M$ at $X^*$.
	\item $\partial \Psi$ is continuous at $X^*$ relative to $\M$.
\end{enumerate}
\end{definition}
Informally, this means that $\Psi|_{\M}$ is smooth at $X^*$,
whereas its value varies sharply in directions departing from
$\M$ in a neighborhood of $X^*$.
It is known \cite{lewis2002active} that for every $X^*\in\S^n_+$, the indicator function
$\delta_{ \S^n_+}$ is partly smooth at $X^*$ with
respect to the manifold
	\begin{equation}
		\label{eq:M2}
	\M_{+}(X^*) \coloneqq \left\{ Y \in \S^n_+ \mid \rank(Y) =
	\rank(X^*) \right\}.
\end{equation}
In the sequel, we extend this partial smoothness result to $\F\cap \S^n_+$, a face of the PSD cone. Note that the standard approach of establishing partial smoothness via transversality cannot be applied here, as the required transversality condition fails to hold. Therefore, we check Definition~\ref{def:ps} directly to verify partial smoothness.
\begin{lemma}[\tb{Partial smoothness of $\Psi$}]\label{lem:partly-smooth-face-psd}
    At every $X^*\in\F\cap\S^n_+$, the indicator function $\Psi(X)\coloneqq\delta_{\F\cap\S^n_+}(X)$ is partly smooth with respect to the manifold 
\[
\M^{\F}_+(X^*) \coloneqq\left\{ Y \in \S^n_+ \mid AY=0,\;\rank(Y) =
	\rank (X^*) \right\}.
\] 
\end{lemma}
\begin{proof}
We verify partial smoothness by checking the three defining conditions in Definition~\ref{def:ps}.
By \eqref{eq-AX-reduce}, there exists a full–row–rank matrix $U\in\R^{(n-m)\times n}$ whose rows form a basis of $\ker(A)$ (hence $AU^\top=0$). Define the linear map 
\[
\Phi:\S^{n-m}\to \F,\qquad \Theta\mapsto U^\top \Theta\, U ,
\]  
wich is a bijective. In particular, for every \(X\in\F\cap\S^n_+\), there exists a unique 
\(\Theta\in\S^{n-m}_+\) such that  
\begin{equation}\label{eq-X-Theta}
    X = \Phi(\Theta) = U^\top \Theta\, U \quad\Longleftrightarrow\quad 
    \Theta=\Phi^{-1}(X)=(UU^\top)^{-1}U X U^\top  (UU^\top)^{-1}.
\end{equation}   
In the special case \(X=X^*\), we write \(\Theta^*:=\Phi^{-1}( X^*) \), so that 
\(X^*=\Phi(\Theta^*)\). We will make frequent use of this correspondence in the subsequent proof.

(1) 
By the linear transformation~\eqref{eq-X-Theta}, we have $$\M^{\F}_+(X^*)=\Phi\bigl(\M_+(\Theta^*)\bigr),$$ where $\M_+(\Theta^*)$ is the 
fixed–rank PSD manifold in $\S^{n-m}$ defined in \eqref{eq:M2}. 
Since $\Phi$ is a bijective linear map, it restricts to a $C^\infty$ diffeomorphism from $\M_+(\Theta^*)$ onto $\M^{\F}_+(X^*)$.
Therefore $\M^{\F}_+(X^*)$ is a smooth embedded submanifold of $\S^{n}$.
Since \(\M^{\F}_+(X^*)\subseteq \F\cap\S^n_+\), the indicator function \(\Psi\) is identically zero on $\M^{\F}_+(X^*)$, and is thus smooth when restricted to $\M^{\F}_+(X^*)$.

(2) The subdifferential of $\Psi$ at $X^*$ is given by
\begin{equation}\label{eq-partial-Psi}
    \partial \Psi (X^*)=N_{\F\cap\S^n_+}(X^*)=\left\{ S \in \S^n \,\mid\, -U S U^\top \succeq 0,\; (U S U^\top) \Theta^* = 0 \right\},
\end{equation}
where we use the expression of the normal cone derived in~\eqref{eq-normal-cone}.
Using a similar strategy as that in Subsection~\ref{subsec-FR}, we obtain
\[
T_{\M^{\F}_+(X^*)}(X^*) = \Big\{ H \in \mathbb{S}^{n} \,\mid\, H = U^\top \widehat{H} U,\; \widehat{H} \in T_{\M_{+}(\Theta^*)}(\Theta^*) \Big\},
\]
where $\M_+(\Theta^*)$ is defined as in \eqref{eq:M2}.
The corresponding normal cone is
\[
N_{\M^{\F}_+(X^*)}(X^*) = \Big\{ S \in \mathbb{S}^{n} \,\mid\, USU^\top\in N_{\M_+(\Theta^*)}(\Theta^*) \Big\}=\Big\{ S \in \mathbb{S}^{n} \,\mid\, (USU^\top) \Theta^*=0 \Big\},
\]
which coincides with the affine span of \( \partial \Psi(X^*) \) in \eqref{eq-partial-Psi}.

(3) 
Since \(\delta_{\S^{n-m}_+}(\Theta)\) is continuous at \(\Theta^*\) 
relative to \(\M_+(\Theta)\), continuity is preserved under the bijection \(\Phi\).  
Indeed,
\[
\Phi^*\big(\partial \Psi (X)\big)
=\left\{ USU^\top \in \S^{n-m} \,\mid\, -U S U^\top \succeq 0,\; (U S U^\top) \Theta = 0 \right\}
=\partial \delta_{ \S^{n-m}_+ }(\Theta),
\]
hence
\(\partial \Psi (X) = (\Phi^*)^{-1}\big( \partial \delta_{\S^{n-m}_+}(\Theta) \big)\),
which proves the claim.

Having verified conditions (1)–(3), we conclude that \(\delta_{\F\cap\S^n_+}\) is partly smooth at \( X^*\) relative to \(\M_+^\F(X^*)\).
\end{proof}
Now we can leverage tools from partial smoothness and manifold identification to show
that our algorithm will find the correct rank for \eqref{eq-CVX-short} that contains a global optimum. The result follows directly from~\cite{lee2022escaping} by analogous arguments and is thus omitted.

\begin{theorem}
	Consider \eqref{eq-CVX-short} with $\Psi$ defined in \eqref{eq:Psi}.
	Consider the two sequences of iterates $\{X^i\}$ and $\{\widehat X^{i}\}$ generated by Algorithm~\ref{alg1}. Then the following hold:
     \begin{enumerate}[leftmargin=*, label=(\roman*)]
	\topsep 0pt
	\partopsep 0pt
	\parsep 0pt
		\item For any subsequence $\{\widehat X^{i_t}\}_t$ such that
			$\widehat X^{i_t} \rightarrow X^*$ for some optimal solution $X^*$ of \eqref{eq-CVX-short}, $X^{i_{t}} \rightarrow X^*$ as
			well.
		\item For the same subsequence as above,
			if $X^*$ satisfies the nondegeneracy condition
			\begin{equation*}
				0 \in \operatorname{relint}\left( \partial F\left( X^* \right)
				\right),
			\end{equation*}
			then there is $t_0 \ge 0$ such that $\rank(X^{i_t+1}) =
			\rank(X^*)$ for all $t \geq t_0$.
	\end{enumerate}
	\label{thm:identify}
\end{theorem}

\subsection{Dual Variable Recovery}\label{subsec-recover-dual}
After solving the ALM subproblem~\eqref{ALM-sub-Y}, we obtain a primal solution $X$. To verify that $X$ satisfies the KKT conditions of both the ALM subproblem~\eqref{ALM-sub-Y} and the original problem~\eqref{eq-SDP-general}, it is necessary to recover the remaining dual variables. While the ALM framework already recovers the dual variables $y$ and $W$, we still need to recover the remaining dual variables $U$ (for $AX=0$) and $S$ (for $X\succeq 0$).
In \cite{RiNNAL}, a recovery procedure based on the low‐rank factorization was proposed, but it can suffer from singularity issues and produce an inaccurate dual solution.  In \cite{RiNNAL+}, an approximate recovery via a PG step was described, which in turn requires solving a nearest‐correlation‐matrix subproblem using the semismooth Newton‐CG method.

By contrast, we present a {direct} and {efficient} recovery scheme tailored to \eqref{eq-SDP-general}, which (i) avoids any nonsmooth issue, and (ii) requires only a single projection onto the null space of $A$.  Specifically, 
the KKT system for \eqref{ALM-sub-Y} can be written as
\begin{equation}\label{KKT-sub}
    \begin{aligned}
        &AX=0,\quad \nabla \phi(X)-A^\top U-U^\top A =S,\\
        &\langle X,{S}\rangle=0,\quad X \succeq 0,\quad S\succeq 0,
    \end{aligned}
\end{equation}
The following theorem shows how to recover the dual variables $(U,S)$ explicitly from $X$, without solving any large linear system.

\begin{theorem}\label{thm-recovery}
     Assume that \eqref{ALM-sub-Y} admits an optimal solution satisfying the KKT conditions~\eqref{KKT-sub}. A matrix $X\in \F\cap\S^n_+$ is a minimizer of \eqref{ALM-sub-Y} if and only if 
    \begin{align}
        \hS:=J(\nabla \phi(X))J&\succeq 0\label{KKT-reduced-S}\\
        X\hS&=0\label{KKT-reduced-XS}
    \end{align}
    where $J := I - A^\top(A A^\top)^{-1}A$ is the orthogonal projector onto the null space of $A$.
    \end{theorem}
\begin{proof}
    (1) Suppose \eqref{ALM-sub-Y} admits an optimal solution $X$ that satisfies the KKT conditions~\eqref{KKT-sub}. Then we have:
    \[
        \hS := J(\nabla \phi(X))J 
        = J(\nabla \phi(X) - A^\top U - U^\top A )J 
        = JSJ \succeq 0,
    \]
    where the second equality follows from $J A^\top = 0$, thus confirming \eqref{KKT-reduced-S}. Moreover, since $X \succeq 0$ and $S \succeq 0$, we may write $X = RR^\top$ and $S = VV^\top$ for some matrices $R$ and $V$. From $\langle X, S \rangle = 0$, it follows that
    \[
        \langle RR^\top, VV^\top \rangle = \|V^\top R\|_F^2 = 0,
    \]
    which implies $XS = RR^\top VV^\top = 0$. Therefore,
    \[
        X\hS = XJSJ = XSJ = 0,
    \]
    where the second equality follows from $XA^\top = 0$. Thus, \eqref{KKT-reduced-XS} is satisfied.\\
	(2) Conversely, suppose \eqref{KKT-reduced-S}, \eqref{KKT-reduced-XS} hold. Define
    \[
        V := (AA^\top)^{-1}A\big(\nabla \phi(X) \big)\left(I - \tfrac{1}{2}A^\top(AA^\top)^{-1}A\right).
    \]
    Then we have
    \[
        S := \nabla \phi(X) - A^\top V - V^\top A =J(\nabla\phi(X) )J=\hS\succeq 0.
    \]
    Moreover, since $X\hS = 0$, we have $\langle X, S \rangle = \Tr(XS)= \Tr(X\hS) = 0$. Therefore, the KKT conditions \eqref{KKT-sub} are satisfied.
\end{proof}
\begin{remark}
    The theorem further implies that for any SDP problem exhibiting facial structure, there must exist at least one KKT point at which the strict complementarity condition does not hold.
\end{remark}
\begin{remark}
Typically, verifying global optimality of an SDP requires solving a linear system of size $mn$ to recover all $mn$ dual variables. In contrast, our recovery strategy involves computing and storing the projector $J =(AA^\top)^{-1}A$ only once, after which it can be efficiently applied to $\nabla \phi(X)$.  This leads to a numerically stable, closed‐form recovery of $(U,S)$, significantly simplifying the global‐optimality verification process.
\end{remark}

\subsection{Efficient projection}

When reformulating \eqref{eq-SDP} into \eqref{eq-SDP-general}, we must enforce the monomial‐consistency constraints ($X\in\L^\cA$), the polyhedral constraints ($X\in\P^\cA$), and the normalization constraint ($\langle H^0,X\rangle=1$). In this subsection, we focus specifically on the case where $\P^\cA = \bN^\cA$ imposes componentwise nonnegativity. We define the affine subspace
\[
\tL^\cA := \L^\cA \cap \left\{X\in\S^\cA \mid \langle H^0,X\rangle = 1\right\}.
\]
Rather than incorporating all constraints directly into the linear system $\Q(X)=b$, which would substantially increase the number of constraints, we instead aggregate them into a single feasibility set:
\[
\P := \tL^\cA \cap \bN^\cA.
\]
This formulation allows us to enforce all constraints through a single projection onto \(\P\). However, since \(\tL^\cA\) involves \(\Omega(n^{2\tau})\) constraints and \(\bN^\cA\) involve \(\Omega(n^{2\tau})\) inequality constraints, a naive projection \(\Pi_{\P}(X)\) requires \(\O(n^{7\tau})\) time, which is computationally prohibitive for large \(n\) or \(\tau\).

To overcome this computational bottleneck, we now show that the projection onto $\P$ can in fact be computed efficiently in \(\mathcal{O}(n^{2\tau})\) time. The key observation is that the projection onto \(\P = \tL^\cA \cap \bN^\cA\) admits a decomposition into sequential projections:

\begin{lemma}\label{lemma-proj-decomp}
The projection onto the intersection \(\P = \tL^\cA \cap \bN^\cA\) satisfies
\begin{equation}\label{proj-decomposition}
    \Pi_{\P} = \Pi_{\bN^\cA} \circ \Pi_{\tL^\cA}.
\end{equation}
\end{lemma}

\begin{proof}
Partition the index set \(\cA \times \cA\) into blocks
\[
B_\gamma := \{\,(\alpha,\beta) \in \cA \times \cA \mid \alpha + \beta = \gamma\,\},
\]
so that each block \(B_\gamma\) consists of pairs $(\alpha,\beta)$ summing to the same 
monomial exponent $\gamma$. The subspace \(\tL^\cA\) enforces equality among entries within each block, and the cone \(\bN^\cA\) enforces nonnegativity componentwise. Therefore, the projection onto \(\P\) decouples across blocks.

Let \(\gamma_0 = (e_1^\tau , e_1^\tau)\) be the unique multi-index where \(H^0\) has support. The normalization constraint \(\langle H^0, X \rangle = 1\) in $\tL^\cA$ implies
\[
\Pi_{\P}(X)|_{B_{\gamma_0}} = 1 = \left(\Pi_{\bN^\cA} \circ \Pi_{\tL^\cA}\right)(X)|_{B_{\gamma_0}}.
\]
For \(\gamma \ne \gamma_0\), the blockwise projection has the closed-form:
\[
\Pi_{\P}(X)|_{B_\gamma}
= \underset{x \ge 0}{\arg\min}\, \sum_{i \in B_\gamma} (x - X_i)^2
= \max\left\{ \frac{1}{|B_\gamma|} \sum_{i \in B_\gamma} X_i,\; 0 \right\}
= \left(\Pi_{\bN^\cA} \circ \Pi_{\tL^\cA}\right)(X)|_{B_\gamma}.
\]
Since this holds for every block, the identity \eqref{proj-decomposition} follows.
\end{proof}

By Lemma~\ref{lemma-proj-decomp}, the projection onto \(\P^\cA\) can be computed efficiently by performing sequential projections onto \(\tL^\cA\) and \(\bN^\cA\). The projection onto \(\bN^\cA\) is simply an entrywise maximum with zero and thus costs \(\O(n^{2\tau})\).
The naive projection onto \(\tL^\cA\) is given by
\[
\Pi_{\tL^\cA}(X) = \left(I - \tilde{\A}^*(\tilde{\A} \tilde{\A}^*)^{-1} \right)\left(\tilde{\A}(X) - \tilde{b}\right),
\]
where \(\tilde{\A}(X) = \tilde{b}\) encodes the linear equalities defining \(\tL^\cA\). Since there are \(\Omega(n^{2\tau})\) such equalities, directly factorizing \(\tilde{\A} \tilde{\A}^*\) incurs a cost of \(\O(n^{6\tau})\), which is again computationally impractical for large-scale instances.

Instead, the structure revealed in the proof of Lemma~\ref{lemma-proj-decomp} offers a more efficient alternative: the projection onto \(\tL^\cA\) reduces to computing the arithmetic mean over each block \(B_\gamma\), and fixing all entries in the block \(B_{\gamma_0}\) to one due to the normalization constraint. Since there are at most \(\O(n^{2\tau})\) distinct index sums \(\gamma\), this averaging step requires only \(\O(n^{2\tau})\) operations. After this projection, all entries corresponding to the same monomial automatically coincide, satisfying the monomial-consistency requirement. In practice, this averaging is implemented using matrix-based operations that group and update entries with the same index sum. The procedure can be further accelerated by exploiting parallel computation or sparse data structures when appropriate.

\section{Extension of RiNNAL-POP for solving moment-SOS relaxation problems}
\label{sec-extension}

RiNNAL-POP, presented in the last section, is designed to solve the polyhedral–SDP relaxation \eqref{eq-SDP} of the polynomial optimization problem \eqref{eq-POP}. However, the framework can be extended to handle the moment–SOS relaxation \eqref{eq-eq-ieq-general-moment} of the general POP formulation \eqref{eq-eq-ieq-general}. In this section, we discuss how such an extension can be achieved.

As mentioned in Subsection~\ref{subsec-FR}, facial constraints naturally arise in moment–SOS relaxations. We specifically exploit this structure in the algorithm design and consider solving the following general $\tau$th-order moment–SOS relaxation of \eqref{eq-eq-ieq-general}:
\begin{equation}\label{eq-eq-ieq-general-facial-moment}
\min\left\{\<Q^0,X\> \;\middle|\;
\begin{aligned}
    &\langle H^0, X \rangle=1,\;
     AX=0,\; \langle Q^i, X \rangle=0,\;i\in [m],\\
    & \M_{h_j}(X)\succeq 0,\; j\in[p],\; X\in \S_+^{\cA_\tau} \cap \mathcal{L}^{\cA_\tau}
\end{aligned}
\right\}.
\end{equation}
This problem can be compactly written as 
\begin{equation}\label{eq-eq-ieq-general-facial-moment-compact}
\min\left\{\<Q^0,X\> \;\middle|\;
\Q(X)=0,\; \M_{h_j}(X)\in \K^{(j)},\; j\in[p],\; X\in\F\cap \P\cap \K
\right\},
\end{equation}
where $\K:=\S^{\cA_\tau}_+$ and $\K^{(j)}:=\S^{\cA_{\tau-\lceil h_j\rceil}}_+$, $\Q(X):=\left[ \langle Q^1,X\rangle,\dots, \langle Q^m,X\rangle\right]^\top$ and 
\begin{equation*}
    \F:=\left\{X\in\S^{\cA_\tau}\mid AX=0\right\},\quad
    \P:=\left\{X\in\L^{\cA_\tau}\mid \langle H^0, X \rangle=1\right\}.
\end{equation*}
To apply the augmented Lagrangian framework, we introduce auxiliary variables $Y\in\K$ and $Y^{(j)}\in\Kj$ for $j\in[p]$, and reformulate \eqref{eq-eq-ieq-general-facial-moment-compact} into the equivalent splitting form below:
\begin{equation}\label{eq-eq-ieq-general-facial-moment-split}
\min\left\{
    \<Q^0,X\>
    +\delta_{\F\cap \K}(X)
    +\delta_{\P}(Y)+\sum_{j=1}^p\delta_{\Kj}(\Yj)
 \;\middle|\;
\begin{aligned}
     &\Q(X)=0,\; X-Y=0,\;\\
     &\M_{h_j}(X)-Y^{(j)}= 0,\; j\in[p] 
\end{aligned}
\right\}.
\end{equation}
RiNNAL-POP can then be extended to solve problem \eqref{eq-eq-ieq-general-facial-moment-split} via the augmented Lagrangian method. The corresponding augmented Lagrangian function is given by
\begin{align*}
    &L_{\sigma}(X,Y,\{Y^{(j)}\};\ y,W,\{W^{(j)}\})
    :=
    \<Q^0,X\> 
    +\frac{\sigma}{2}\|\Q(X)-\sigma^{-1}y\|^2
    +\frac{\sigma}{2}\|X-Y-\sigma^{-1}W\|^2\\
    &\qquad\qquad
    +\frac{\sigma}{2}\sum_{j=1}^p \|\M_{h_j}(X)-Y^{(j)}-\sigma^{-1}W^{(j)}\|^2 
    -\frac{1}{2\sigma}\|y\|^2-\frac{1}{2\sigma}\|W\|^2-\frac{1}{2\sigma}\sum_{j=1}^p\|W^{(j)}\|^2.
\end{align*}
Given the initial penality parameter $\sigma_0>0$, dual variable $y_0\in \R^{d}$, $W_0\in\K$ and $W^{(j)}_0\in\Kj$, the augmented Lagrangian method performs the following steps at the $(k+1)$-th iteration:
\begin{align*}
    (X_{k+1},Y_{k+1},\{Y^{(j)}_{k+1}\})&=\operatorname{argmin} \left\{L_{\sigma_k}(X,Y,\{Y^{(j)}\};\ y_k,W_k,\{W^{(j)}_{k}\})  \;\middle|\; \begin{aligned}
        &X\in {\F\cap \K},\; Y\in\P,\\ &Y^{(j)}\in \Kj, \; j\in[p]
    \end{aligned} \right\},\\
    y_{k+1}&=y_k-\sigma_k \Q (X_{k+1}),\notag\\
    W_{k+1}&=W_k-\sigma_k(X_{k+1}-Y_{k+1}),\notag\\
    W^{(j)}_{k+1}&=W^{(j)}_k-\sigma_k(\M_{h_j}(X_{k+1})-Y^{(j)}_{k+1}),\quad j\in[p],\notag
\end{align*}
where $\sigma_k \uparrow \sigma_{\infty} \leq+\infty$ are positive penalty parameters.
Let $\widetilde y$, $\widetilde W\in \bS^{n}$ and $\widetilde W^{(k)}$ be fixed. The ALM primal subproblem can be expressed as:
\begin{equation}
    \min\ \Big\{L_{\sigma}(X,Y,\{\Yj\};\ \widetilde y,\widetilde W,\{\widetilde W^{(k)}\} ) \mid X\in {\F\cap \K},\; Y\in\P,\; Y^{(j)}\in \Kj, \; j\in[p] \Big\}.\label{mom-ALM-sub-XY-fix}
\end{equation}
In \eqref{mom-ALM-sub-XY-fix}, we can first minimize with respective to $Y\in\P$ and $Y^{(j)}\in \Kj$ to get the following convex optimization problem related only to $X$:
\begin{equation}\label{mom-ALM-sub-Y}
    \min\left\{
    \begin{aligned}
        \phi (X) &:= \langle C,X\rangle 
        + \frac{\sigma}{2} \|\sigma^{-1}\wt y-(\Q(X)-b)\|^2 \\
        &+ \frac{\sigma}{2}\|X-\sigma^{-1}\wt W-\Pi_{\P}(X-\sigma^{-1}\wt W)\|^2\\
        &+ \frac{\sigma}{2}\sum_{j=1}^{p}\|\Pi_{\Kj}( \sigma^{-1}\wt W^{(j)} - \Mj (X))\|^2
    \end{aligned}
    \;\middle|\; X\in {\F\cap \K} \right\},
\end{equation}
where we used the Moreau decomposition theorem in \cite{moreau1962decomposition}, which states that $X=\Pi_{\mathcal{C}}(X)-\Pi_{\mathcal{C}^*}(-X)$ for any closed convex cone $\mathcal{C}$.
Once we obtain the optimal solution $\wt X$ of \eqref{mom-ALM-sub-Y}, we can recover the optimal solution $\wt Y=\Pi_P(\wt X-\sigma^{-1}\wt W)$ and $\wt Y^{(j)}=\Pi_{\Kj}(\Mj(\wt X)-\sigma^{-1}\wt W^{(j)})$. 

From here, we can apply the same techniques 
 for solving \eqref{ALM-sub-Y} in Section~\ref{sec-alg} to 
solve the reduced subproblem \eqref{mom-ALM-sub-Y}.

\section{Numerical experiments}\label{sec-exp}

In this section, we conduct numerical experiments to illustrate the effectiveness of RiNNAL-POP for solving polyhedral-SDP relaxation problems of the form \eqref{eq-SDP-general}. All experiments are performed using {\sc Matlab} R2023b on a workstation equipped with Intel Xeon E5-2680 (v3) processors and 96GB of RAM.

\medskip

\noindent\textbf{Baseline Solvers}.  
We compare RiNNAL-POP with SDPNAL+ \cite{SDPNAL,SDPNALp1,SDPNALp2}. While other ALM-based and low-rank SDP solvers exist, we do not include them in our comparison, as they are either incompatible with \eqref{eq-SDP-general} or 
are clearly outperformed by SDPNAL+ in preliminary tests. 

\medskip

\noindent\textbf{Stopping Conditions}. Based on the KKT conditions \eqref{KKT} for \eqref{eq-SDP-general}, we define the following relative KKT residuals to assess the accuracy of the solution:
\small{
\begin{align*}
\operatorname{R_p}:=
\max\left\{
\frac{{\|\Q(X)- b\|}}{1+\| b \|},\;
\frac{\|X-Y\|}{1+\|X\|+\|Y\|}
\right\},\;\;
\operatorname{R_d} :=  
\frac{\|\Pi_{\S_+^{n+1}}(S)\|}{1+\|S\|},\;\;
\operatorname{R_{c}} := 
\frac{|\langle X, S\rangle|}{1+\|X\|+\|S\|}
.
\end{align*}
}
\normalsize
\noindent
We omit residuals associated with the constraints
$X\in \F\cap\S^{n+1}_+$, $Y\in\P$, and $W\in N_\P(Y)$,
as the first is inherently satisfied by the low-rank factorization of $X$, while the latter two are automatically satisfied by the design of the ALM iteration.
For a given tolerance $\tol> 0$, RiNNAL-POP terminates when the maximum residual satisfies $\operatorname{R_{max}}:=\max\{\operatorname{R_p},\operatorname{R_d},\operatorname{R_c}\}<\tol$ or the maximum time limit $\timelimit$ is reached. In our experiments, we set $\tol = 10^{-6}$ and $\timelimit = 3600\tt{(secs)}$ for all solvers.

\medskip

\noindent
\textbf{Implementation}. 
In RiNNAL-POP, each augmented Lagrangian subproblem is approximately solved by using a two-phase scheme that alternates between a low-rank phase and a convex-lifting phase. Both phases employ the projected gradient (PG) method, but with different step selection strategies: the low-rank phase uses Barzilai–Borwein step sizes combined with a nonmonotone line search strategy~\cite{RiNNAL, iannazzo2018riemannian, gao2021riemannian, tang2023feasible}, while the convex-lifting phase uses a fixed stepsize of $1/\sigma_k$.
The penalty parameter $\sigma_k$ is initialized at $\sigma_0 = 1$ and updated dynamically: it is doubled if $\operatorname{R_p}/\|\nabla f_r(R)\| \geq 10$, and halved otherwise. The initial rank is set to $r_0 = \min\{200, \lceil n/5 \rceil\}$, and the initial point $R_0$ is selected randomly from the feasible set $\mathcal{M}_{r_0}$. In each ALM iteration, the low-rank phase performs at most 50 PG iterations, followed by a single PG step in the convex-lifting phase.

\medskip

\noindent\textbf{Table Notations.} 
We use `-' to indicate that an algorithm did not achieve the required tolerance $\tol$ within the maximum time limit $\timelimit$. 
Results are reported only when the maximum KKT residual is below $10^{-2}$. A superscript ``$\dagger$'' on the KKT residual indicates that only moderate accuracy was achieved.
In the first column, we list the polynomial‐problem dimension $n$, the SDP‐problem dimension $N$, and the number of equality constraints $m$, respectively.
For the column labeled ``Iteration'' associated with SDPNAL+, the first entry denotes the number of outer iterations, the second entry denotes the total number of semismooth Newton inner iterations, and the third indicates the total number of ADMM+ iterations. Similarly, for the column labeled ``Iteration'' associated with RiNNAL-POP, the first entry corresponds to the number of ALM iterations, the second denotes the total number of projected gradient descent iterations in the low-rank phase, and the third reports the total number of convex lifting steps.
The column labeled ``Rank'' indicates the rank of the final iterate, computed using a threshold of $10^{-6}$ for both RiNNAL-POP and SDPNAL+.
The ``Objective'' column denotes the value of the objective function, while the total computation time is listed under ``Time''. We denote RiNNAL-POP by {RPOP} for brevity.

\medskip

\noindent\textbf{Tested Problems.}
We summarize below the polynomial‐optimization models examined in our experiments, together with their decision‐variable dimension \(n\), POP order (Deg), relaxation order (Relax), the form of equalities (Eq), the form of inequalities (Ineq), the presence of nonnegativity constraints ($\geq 0$), and the observed runtime ratio SDPNAL+/RPOP (Speedup).

\begin{longtable}{@{}lcccllcr@{}}
\caption{Summary of tested problem classes.} \label{table:problem_summary} \\
\toprule
Problem & $n$ & Deg & Relax & Eq & Ineq & $\ge 0$ & Speedup \\
\midrule
\endfirsthead
\toprule
Problem & $n$ & Deg & Relax & Eq & Ineq & $\ge 0$ & Speedup \\
\midrule
\endhead
\eqref{eq-STQP}  &$1000-6000$ & $2$ & $1$ & Simplex & - & $\checkmark$ & $\phantom{0}34-36\phantom{0}$ \\
\eqref{eq-STQP}  &$20-80$ & $2$ & $2$ & Simplex & -  & $\checkmark$ & $\phantom{00}3-6\phantom{00}$\\
\eqref{eq-BIQ}  &$20-60$ & $2$ & $2$ & Binary & - & $\checkmark$ & $\phantom{0}10-25\phantom{0}$\\
\eqref{eq-MBP}  &$30-60$ & $2$ & $2$ & Equipartition & - & $\checkmark$ & $\phantom{0}21-30\phantom{0}$\\
\eqref{eq-MQKP-slack} &$20-40$ & $2$ & $2$ & Binary & Knapsack & $\checkmark$ & $\phantom{00}10-30\phantom{0}$\\
\eqref{eq-QB-slack} &$20-80$ & $4$ & $2$ & - & Ball & - & $\phantom{00}6-16\phantom{0}$ \\
\eqref{eq-KM} &$20-60$ & $4$ & $2$ & Mean–variance & - & $\checkmark$ & $200-278$\\
\eqref{eq-CST}  &$1000-6000$ & $2$ & $1$ & Simplex & - & $\checkmark$ & $\phantom{0}18-500$\\
\eqref{eq-CST}  &$20-60$ & $4$ & $2$ & Simplex & - & $\checkmark$ & $\phantom{0}13-173$\\
(NSTF) &$20-80$ & $3,4$ & $2$ & Spherical & - & $\checkmark$ & $\phantom{00}2-7\phantom{00}$\\
\eqref{eq-NTF-reduced} &$10-25$ & $3,4$ & $2$ & Block spherical & - & $\checkmark$ & $\phantom{00}2-10\phantom{0}$\\
\bottomrule
\end{longtable}

\subsection{Comparison of relaxation quality and solver performance}\label{subsec-tightness}

We compare the relaxation bound quality and computational efficiency of different SDP-based relaxations for the standard quadratic problem:
\begin{equation}\label{eq-STQP-tightness}
v^*=\min\left\{x^\top Q x \mid\  e^\top x = 1,\ x\in\R^n_+\right\}.\tag{StQP}
\end{equation} 
Specifically, we consider four SDP-based relaxations distinguished by the inclusion of polyhedral (nonnegativity) constraints, RLT constraints, and PSD localizing matrices constraints. These relaxations are summarized in Table~\ref{tab:relaxation-types}.
\begin{table}[ht!]
    \centering
    \renewcommand{\arraystretch}{1.25}
    \begin{tabular}{|c|c|c|c|l|}
        \hline
         Relaxation&  $X\geq 0$& RLT & 
         $\begin{array}{l}\mbox{Localizing}\\ \mbox{PSD}\end{array}$ 
         & Remarks \\
         \hline
         Poly-SDP& $\checkmark$ & - & - & Polyhedral SDP relaxation~\eqref{eq-SDP} \\ 
         \hline
         Poly-SDP-RLT& $\checkmark$ &$\checkmark$  & - & Add RLT (cf.\ Figure \ref{fig:RLT-constraints}
  in Subsection~\ref{subsec-FR}) \\
         \hline
         Mom-SOS& - & $\checkmark$ & $\checkmark$ & Moment-SOS relaxation~\eqref{eq-eq-ieq-general-moment}\\
         \hline
         Poly-Mom-SOS& $\checkmark$ &$\checkmark$  & $\checkmark$ & Add $X\geq 0$\\
         \hline
    \end{tabular}
    \caption{Overview of SDP-based relaxations.}
    \label{tab:relaxation-types}
\end{table}

The relaxation quality is evaluated using the relative gap:
\[
\%\mathrm{gap}
\;=\;
\frac{v^* - v}{\max\{1,\lvert v^*\rvert\}}
\;\times\;
100\%,
\]
where $v^*$ is the optimal value of the StQP and $v$ is the relaxation value.
We evaluate both tightness and computational performance on two types of test instances. To ensure a fair comparison, the reported runtimes exclude 
the time to generate the SDP data
and preprocessing for all solvers.

\medskip
\noindent\tb{Random instances.}
We generate random symmetric matrices $Q\in\mathbb{S}^n$ with entries sampled from $\mathcal{N}(0,1)$, for $n \in \{5, 10, 15, 20\}$. Each relaxation is solved at orders $\tau=1$ and $\tau=2$, using RiNNAL-POP (RPOP) with a solver tolerance of \(\mathrm{tol}=10^{-6}\). 
As the package Gloptipoly~\cite{henrion2009gloptipoly} is commonly used to solve Mom-SOS relaxatons 
of POPs, 
for benchmarking purpose,  we also report results obtained by using GloptiPoly to solve 
the Mom-SOS relaxation here
with the same tolerance.
The resulting relative gaps and computational times are shown in Table~\ref{tab-tightness-random} and Figure~\ref{fig-StQP-random}. In Figure~\ref{fig-StQP-random}, we omit the results for the second-order Poly–SDP–RLT relaxation, as the first-order relaxation is already tight. We omit the results for the first-order Poly–Mom–SOS relaxation, as it is equivalent to the first-order Poly–SDP–RLT relaxation. 
\begin{figure}[ht!]
  \centering
  \begin{minipage}[c]{0.45\textwidth}
    \centering
    \renewcommand{\arraystretch}{1.25} 
    \begin{tabular}{|c|c|c|}
      \hline
      Order & 1st & 2nd \\
      \hline
      Poly–SDP       & $\infty$ & $\infty$ \\
      Poly–SDP–RLT   & $0$      & $0$      \\
      Mom–SOS        & $\infty$ & $0$      \\
      Poly–Mom–SOS   & $0$      & $0$      \\
      \hline
    \end{tabular}
    \captionof{table}{Relative gaps (\%) of the four relaxations for the random instance.}
    \label{tab-tightness-random}
  \end{minipage}
  \hfill
  \begin{minipage}[c]{0.5\textwidth}
    \centering
    \includegraphics[width=\linewidth]{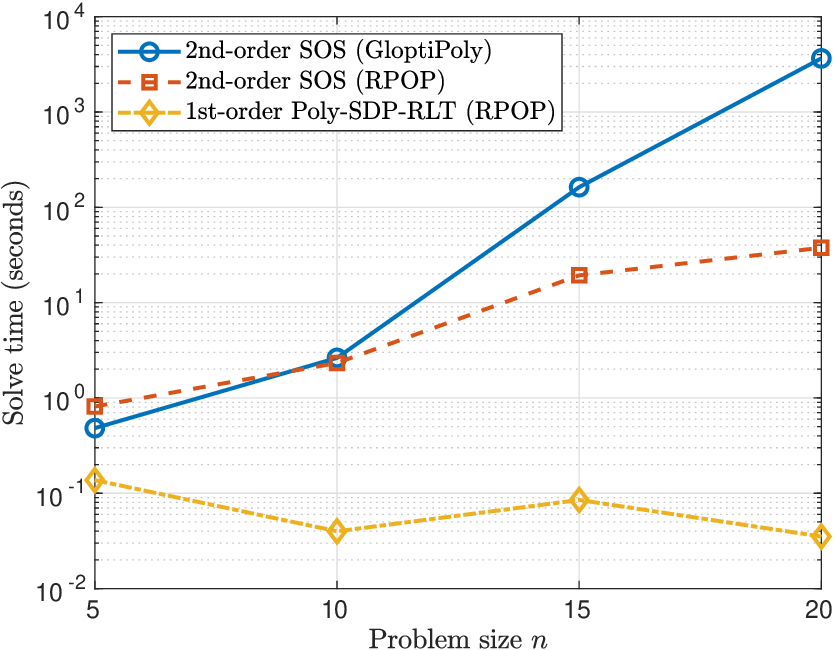}
    \captionof{figure}{Computational time comparison between GloptiPoly and RPOP.}
    \label{fig-StQP-random}
  \end{minipage}
\end{figure}

\medskip
\noindent\tb{Extended Horn matrix.}
For any odd integer $n \ge 5$, Johnson and Reams~\cite[Section~4]{johnson2008constructing} define the matrix $Q_n \in \mathbb{R}^{n\times n}$ by
\[
(Q_n)_{ij} \;=\;
\begin{cases}
-1, & |i-j|=1 \ \text{or}\ \{i,j\}=\{1,n\},\\[0.2em]
1, & \text{otherwise},
\end{cases}
\]
which is a copositive example with known optimal value $v^* = 0$, attained by any $x = (e_i+e_j)/2$, $1 \le i < j \le n$. When $n=5$, $Q_5$ coincides with the classical Horn matrix.
While many of the tested random instances are tight at first- or second-order relaxation, the extended Horn matrix remains a nontrivial case requiring higher-order relaxations to close the gap. It thus serves as a meaningful benchmark for assessing the strength of different SDP relaxations.
For this test, we set $n=21$ and solve each relaxation at orders $\tau=1,2$; in addition, we test GloptiPoly3 for the Mom–SOS relaxation as a benchmark.
The resulting gaps and runtimes are reported in Table~\ref{tab-tightness-horn}.

{
\begin{table}[h!]
  \centering
  \renewcommand{\arraystretch}{1.25}
  \small 
  \begin{tabular}{|c|rr|rr|}
    \hline
    & \multicolumn{2}{c|}{1st‑order} 
    & \multicolumn{2}{c|}{2nd‑order}  \\
    \cline{2-5}
    Relaxation 
      & Gap (\%) & Time (s) 
      & Gap (\%) & Time (s)  \\
    \hline
    Poly–SDP         
      & $\infty$      & - 
      & $\infty$      & -        \\
    Poly–SDP–RLT    
      & $0.562$      & 0.3
      & $\phantom{0}0.280$ & \phantom{0}4.6 \\
      Mom–SOS (RPOP)       
      & $\infty$      & - 
      & $1.049$       & 1875.4\\
    Mom–SOS (Gloptipoly)         
      & $\infty$      & - 
      & $1.046$       & 6394.0\\
    Poly–Mom–SOS     
      & $0.562$      & 0.3
      & $\phantom{0}0.000$ & \phantom{0}245.7  \\
    \hline
  \end{tabular}
  \caption{Relative gaps and run times for the extended Horn instance ($n=21$).}
  \label{tab-tightness-horn}
\end{table}
}

\medskip

The numerical results from both the random and extended Horn instances reveal the following key insights into the relaxation quality and computational performance of different SDP-based relaxations:
\begin{enumerate}
    \item \tb{RLT and nonnegativity constraints enhance relaxation strength.} 
    Without RLT or nonnegativity constraints, the Poly–SDP and Mom–SOS relaxations remain unbounded at low orders, as evidenced by both the random and extended Horn instances. Incorporating RLT constraints alone (Poly–SDP–RLT) produces finite bounds — it is exact at first order on random matrices and reduces the extended Horn matrix's gap to 0.280\% at second order. Even more dramatically, adding nonnegativity constraint to the second order Mom–SOS relaxation (Poly–Mom–SOS) yields an exact bound for the extended Horn example. 
    In contrast, the standard second order Mom-SOS relaxation only achieves
    a gap of 1.046\%. 
    
    The above observations show that RLT and nonnegativity constraints complement moment information, closing gaps at lower orders and enhancing both tightness and efficiency. This performance advantage motivates our focus on the Poly–SDP–RLT relaxation in subsequent sections.
    \item \tb{Poly–SDP–RLT balances relaxation quality and computational cost.} 
    Among the relaxations tested, Poly–SDP–RLT consistently achieves tight bounds at low computational cost. 
    Remarkably, for the more challenging extended Horn matrix, the first-order Poly–SDP–RLT relaxation can produce a tighter bound than the second-order Mom–SOS relaxation, with a speedup of about 
$20,000$ times.
On random instances, Poly–SDP–RLT is exact already at first order. 
In contrast, Mom–SOS must resort to higher‐order moment matrices and incurs substantially longer runtimes. For instance, as shown in Figure~\ref{fig-StQP-random}, at $n=20$, RPOP solves the first‐order Poly–SDP–RLT relaxation about $100{,}000$ times faster than GloptiPoly for solving the second‐order Mom–SOS relaxation. This dramatic speedup highlights the efficiency gains from polyhedral strengthening, making Poly–SDP–RLT a scalable and practical alternative to high‐order moment–SOS relaxations.
    \item \tb{RPOP outperforms GloptiPoly in both efficiency and scalability.} 
    Although originally designed for Poly-SDP–RLT, RPOP extends naturally to moment–SOS problems as we have shown in \ref{sec-extension}, 
    and it consistently outperforms GloptiPoly across nearly all tested sizes. For example, on random instances with $n=20$, RPOP solves the second‐order Mom–SOS relaxation approximately $100$ times faster than GloptiPoly (Figure~\ref{fig-StQP-random}). This performance gap widens with increasing $n$, reflecting the superior scalability of RPOP’s low‐rank augmented Lagrangian approach compared to interior‐point methods.
\end{enumerate}

\subsection{Standard quadratic programming}\label{subsec-StQP}

Consider the following StQP problem mentioned in Subsection~\ref{subsec-tightness}:
\begin{equation}\label{eq-STQP}
\min\left\{x^\top Q x :\  e^\top x = 1,\ x\in\R^n_+\right\}.\tag{StQP}
\end{equation} 
To obtain a tractable convex approximation of \eqref{eq-STQP}, we employ the Poly-SDP-RLT relaxation as detailed in Subsection~\ref{subsec-tightness}.
We test on two classes of instances:
\begin{enumerate}
  \item \textbf{Random Gaussian.}  For \(n\in\{1000,2000,4000,6000\}\), each entry of \(Q_{ij}\) is sampled independently from the standard normal distribution \(\mathcal{N}(0,1)\). Empirically, these matrices often admit extremely sparse solutions $x$, and the first‐order Poly‐SDP‐RLT relaxation typically yields the exact global optimum.
  \item \textbf{COP‐derived.}  For $n\in\{20,40,60,80\}$, we generate \(Q\) following the method in~\cite{bomze2024tighter}, which produces instances with a unique minimizer supported on exactly $n/2$ variables.
These problems are more challenging, and the first-order relaxation is generally inexact. We therefore apply the second-order Poly-SDP-RLT relaxation, which empirically yields significantly tighter bounds.
\end{enumerate}

The computational results are summarized in Tables~\ref{table:StQP1} and~\ref{table:StQP2}.

{\small
\begin{longtable}{@{}l l l r l r r@{}}
\caption{Computational results for StQP problems (relaxation order $\tau =1$).}\label{table:StQP1} \\
\toprule
{(n, N), m} & \multicolumn{1}{l}{Algorithm} & \multicolumn{1}{l}{Iteration} & \multicolumn{1}{c}{Rank} & \multicolumn{1}{c}{$R_{\max}$} & \multicolumn{1}{c}{Objective} & {Time} \\
\midrule
\endfirsthead
\toprule
{(n, N), m} & \multicolumn{1}{l}{Algorithm} & \multicolumn{1}{l}{Iteration} & \multicolumn{1}{c}{Rank} & \multicolumn{1}{c}{$R_{\max}$} & \multicolumn{1}{c}{Objective} & {Time} \\
\midrule
\endhead
\multirow[t]{2}{*}[0pt]{\parbox[t]{2cm}{(1000,1001)\\m=1002}} & RPOP & 25, 1238, 25 & 1 & 4.32e-07 & -4.2604198e-03 & 15.8 \\
 & SDPNAL+ & 136, 146, 4485 & 1 & 4.78e-07 & -4.1609498e-03 & 550.4 \\
\addlinespace
\multirow[t]{2}{*}[0pt]{\parbox[t]{2cm}{(2000,2001)\\m=2002}} & RPOP & 27, 1400, 27 & 2 & 8.66e-07 & -2.0202661e-03 & 100.2 \\
 & SDPNAL+ & 71, 155, 3397 & 4 & 4.28e-05$^{\dagger}$ & -2.0064321e-03 & 3601.3 \\
\addlinespace
\multirow[t]{2}{*}[0pt]{\parbox[t]{2cm}{(4000,4001)\\m=4002}} & RPOP & 29, 1500, 29 & 1 & 5.67e-07 & -1.1300123e-03 & 957.6 \\
 & SDPNAL+ & 12, 30, 400 & 56 & 8.35e-05$^{\dagger}$ & -7.3488183e-06 & 3653.6 \\
\addlinespace
\multirow[t]{2}{*}[0pt]{\parbox[t]{2cm}{(6000,6001)\\m=6002}} & RPOP & 44, 2250, 44 & 2 & 1.99e-05$^{\dagger}$ & -8.3325597e-04 & 3628.8 \\
 & SDPNAL+ & 0, 0, 117 & 47 & 1.75e-03$^{\dagger}$ & -2.7447386e-05 & 3653.8 \\
\addlinespace
\bottomrule
\end{longtable}
}

As shown in Table~\ref{table:StQP1}, RPOP successfully solves all random Gaussian instances to the required accuracy, consistently producing low-rank solutions. On the largest instance $(n=6000)$, although RPOP reaches the time limit, it achieves a KKT residual over 100 times smaller than SDPNAL+, whose solution is both inaccurate and of high rank. Across all tested sizes, RPOP is up to 35 to 40 times faster than SDPNAL+.

{\small
\begin{longtable}{@{}l l l r l r r@{}}
\caption{Computational results for StQP problems (relaxation order $\tau=2$)).}\label{table:StQP2} \\
\toprule
{(n, N), m} & \multicolumn{1}{l}{Algorithm} & \multicolumn{1}{l}{Iteration} & \multicolumn{1}{c}{Rank} & \multicolumn{1}{c}{$R_{\max}$} & \multicolumn{1}{c}{Objective} & {Time} \\
\midrule
\endfirsthead
\toprule
{(n, N), m} & \multicolumn{1}{l}{Algorithm} & \multicolumn{1}{l}{Iteration} & \multicolumn{1}{c}{Rank} & \multicolumn{1}{c}{$R_{\max}$} & \multicolumn{1}{c}{Objective} & {Time} \\
\midrule
\endhead
\multirow[t]{2}{*}[0pt]{\parbox[t]{2cm}{(20,231)\\m=21022}} & RPOP & 23, 1200, 23 & 8 & 3.12e-07 & -4.4326119e-02 & 1.5 \\
 & SDPNAL+ & 0, 0, 338 & 11 & 9.43e-07 & -4.4340416e-02 & 4.2 \\
\addlinespace
\multirow[t]{2}{*}[0pt]{\parbox[t]{2cm}{(40,861)\\m=270642}} & RPOP & 27, 1400, 27 & 8 & 8.07e-07 & -4.7130297e-02 & 16.5 \\
 & SDPNAL+ & 0, 0, 422 & 12 & 9.48e-07 & -4.7150853e-02 & 65.5 \\
\addlinespace
\multirow[t]{2}{*}[0pt]{\parbox[t]{2cm}{(60,1891)\\m=1268862}} & RPOP & 24, 1250, 24 & 8 & 4.47e-07 & -4.8533814e-02 & 111.8 \\
 & SDPNAL+ & 0, 0, 491 & 12 & 9.51e-07 & -4.8557752e-02 & 475.7 \\
\addlinespace
\multirow[t]{2}{*}[0pt]{\parbox[t]{2cm}{(80,3321)\\m=3855682}} & RPOP & 24, 1250, 24 & 8 & 4.14e-07 & -4.8304554e-02 & 623.4 \\
 & SDPNAL+ & 21, 39, 432 & 21 & 2.67e-05$^{\dagger}$ & -4.9092914e-02 & 3604.2 \\
\addlinespace
\bottomrule
\end{longtable}
}

Table~\ref{table:StQP2} presents results for the more challenging COP-derived instances using the second-order Poly-SDP-RLT relaxation. For problem sizes up to $n = 60$, where the lifted SDP already involves tens to hundreds of thousands of constraints due to the RLT and nonnegativity terms, both solvers reach the required accuracy. However, RPOP consistently yields lower-rank solutions and runs 3 to 4 times faster than SDPNAL+. On the largest instance ($n = 80$), which leads to over 3.8 million linear equality and 11 million nonnegativity constraints, SDPNAL+ fails to reach the required accuracy within the time limit. In contrast, RPOP achieves the required accuracy within 10 minutes, over 6 times faster than SDPNAL+.

\subsection{Binary quadratic programs}
Consider the following binary quadratic program:
\begin{equation}\label{eq-BIQ-org}
\min\left\{x^\top Q x+c^\top x \mid x\in\{0,1\}^n\right\},
\end{equation}
where $Q\in\S^n$, $c\in \R^n$. 
This problem can be equivalently reformulated as a polynomial optimization problem in the standard form of~\eqref{eq-POP} to enable SDP-based relaxation:
\begin{equation}\label{eq-BIQ}
\min\left\{x^\top Q x+c^\top x \mid  x_i^2-x_i=0,\; \forall i\in[n],\; x\in\R_+^n\right\}.\tag{BIQ}
\end{equation}
As demonstrated in \cite{RiNNAL}, the first‐order Poly‐SDP‐RLT relaxation of \eqref{eq-BIQ} fails to provide exact bounds. To address this, we apply the second-order Poly-SDP-RLT relaxation, which empirically yields exact bounds across all tested instances, thereby significantly enhancing the relaxation quality.  For each $n\in\{20,40,60\}$, we generate random instances of \eqref{eq-BIQ} by sampling each $Q_{ij}$ and $c_i$ independently from the standard normal distribution $\mathcal{N} (0,1)$. Each instance is then solved using both RPOP and SDPNAL+, and the results are summarized in Table~\ref{table:BIQ}.

{\small
\begin{longtable}{@{}l l l r l r r@{}}
\caption{Computational results for BIQ problems (relaxation order $\tau=2$).}\label{table:BIQ} \\
\toprule
{(n, N), m} & \multicolumn{1}{l}{Algorithm} & \multicolumn{1}{l}{Iteration} & \multicolumn{1}{c}{Rank} & \multicolumn{1}{c}{$R_{\max}$} & \multicolumn{1}{c}{Objective} & {Time} \\
\midrule
\endfirsthead
\toprule
{(n, N), m} & \multicolumn{1}{l}{Algorithm} & \multicolumn{1}{l}{Iteration} & \multicolumn{1}{c}{Rank} & \multicolumn{1}{c}{$R_{\max}$} & \multicolumn{1}{c}{Objective} & {Time} \\
\midrule
\endhead
\multirow[t]{2}{*}[0pt]{\parbox[t]{2cm}{(20,231)\\m=20791}} & RPOP & 36, 7332, 8 & 1 & 4.43e-07 & -4.3222829e+01 & 5.5 \\
 & SDPNAL+ & 85, 275, 2450 & 1 & 1.36e-07 & -4.3222841e+01 & 55.0 \\
\addlinespace
\multirow[t]{2}{*}[0pt]{\parbox[t]{2cm}{(40,861)\\m=269781}} & RPOP & 73, 14701, 15 & 1 & 1.83e-07 & -1.0323647e+02 & 92.2 \\
 & SDPNAL+ & 186, 1716, 10948 & 1 & 5.85e-07 & -1.0323635e+02 & 2377.0 \\
\addlinespace
\multirow[t]{2}{*}[0pt]{\parbox[t]{2cm}{(60,1891)\\m=1266971}} & RPOP & 60, 12200, 12 & 1 & 3.63e-07 & -1.7504234e+02 & 517.9 \\
 & SDPNAL+ & 28, 236, 1000 & 848 & 2.54e-03$^{\dagger}$ & -1.5375472e+02 & 3604.0 \\
\addlinespace
\bottomrule
\end{longtable}
}

As shown in Table~\ref{table:BIQ}, RPOP solves all instances to the required accuracy and consistently yields rank-one solutions, confirming the empirical tightness of the second-order Poly-SDP-RLT relaxation. In contrast, SDPNAL+ exhibits deteriorating performance as problem size increases. For the medium-scale instance ($n = 40$), RPOP is over 25 times faster. For the largest instance ($n = 60$), RPOP converges within 10 minutes with an accurate, rank-one solution, whereas SDPNAL+ fails to reach the required accuracy within the one-hour time limit and terminates with an 
approximate solution with poor KKT residual. These results highlight the robustness and efficiency of RPOP in producing low-rank, high-quality solutions for increasingly large-scale Poly-SDP-RLT relaxations of binary quadratic problems.

\subsection{Minimum bisection problem}
Consider the minimum bisection problem over undirected graphs with $n$ nodes, formulated as the following binary quadratic program:
\begin{equation}\label{eq-MBP}
\min\left\{x^\top L x :\  e^\top x = \frac{n}{2},\ x\in\{0,1\}^n\right\},\tag{MBP}
\end{equation}
where \(L\) is the Laplacian of a random Erdos–Renyi graph \(G(n,p)\) with edge probability \(p=0.5\), and the equality constraint enforces an equipartition of the vertices. 
For each $n \in \{ 30, 40, 50, 60\}$, we generate a single random instance and solve its second-order Poly-SDP-RLT relaxation, since the first-order relaxation fails to provide tight bounds. Table~\ref{table:MB} summarizes the performance of RPOP and SDPNAL+.

{\small
\begin{longtable}{@{}l l l r l r r@{}}
\caption{Computational results for MB problems (relaxation order 
$\tau=2$).}\label{table:MB} \\
\toprule
{(n, N), m} & \multicolumn{1}{l}{Algorithm} & \multicolumn{1}{l}{Iteration} & \multicolumn{1}{c}{Rank} & \multicolumn{1}{c}{$R_{\max}$} & \multicolumn{1}{c}{Objective} & {Time} \\
\midrule
\endfirsthead
\toprule
{(n, N), m} & \multicolumn{1}{l}{Algorithm} & \multicolumn{1}{l}{Iteration} & \multicolumn{1}{c}{Rank} & \multicolumn{1}{c}{$R_{\max}$} & \multicolumn{1}{c}{Objective} & {Time} \\
\midrule
\endhead
\multirow[t]{2}{*}[0pt]{\parbox[t]{2cm}{(30,496)\\m=107137}} & RPOP & 123, 6200, 25 & 2 & 9.38e-09 & 3.7600000e+02 & 17.6 \\
 & SDPNAL+ & 45, 235, 2900 & 2 & 4.18e-07 & 3.7599983e+02 & 365.4 \\
\addlinespace
\multirow[t]{2}{*}[0pt]{\parbox[t]{2cm}{(40,861)\\m=305082}} & RPOP & 32, 6463, 7 & 2 & 7.60e-07 & 6.3600021e+02 & 64.2 \\
 & SDPNAL+ & 82, 522, 5233 & 2 & 2.23e-07 & 6.3600002e+02 & 1878.5 \\
\addlinespace
\multirow[t]{2}{*}[0pt]{\parbox[t]{2cm}{(50,1326)\\m=697477}} & RPOP & 248, 49800, 50 & 5 & 1.22e-07 & 9.5600001e+02 & 1014.8 \\
 & SDPNAL+ & 44, 652, 2827 & 1316 & 1.24e-03$^{\dagger}$ & 9.6287415e+02 & 3600.6 \\
\addlinespace
\multirow[t]{2}{*}[0pt]{\parbox[t]{2cm}{(60,1891)\\m=1382322}} & RPOP & 385, 77200, 77 & 52 & 3.36e-04$^{\dagger}$ & 1.4479773e+03 & 3606.9 \\
 & SDPNAL+ & 25, 96, 777 & 1888 & 4.02e-03$^{\dagger}$ & 1.4615768e+03 & 3600.9 \\
\addlinespace
\bottomrule
\end{longtable}
}

As shown in Table~\ref{table:MB}, RPOP solves all instances up to $n = 50$ to the required accuracy, while SDPNAL+ fails to reach the target accuracy for $n = 50$, returning a high-rank solution with poor KKT residuals. For $n = 30$ and $n = 40$, RPOP is 20–30 times faster than SDPNAL+. At $n = 60$, neither solver meets the accuracy requirement within the time limit, but RPOP produces a lower-rank solution with a KKT residual that is an order of magnitude smaller than that of SDPNAL+.

\subsection{Multiple quadratic knapsack problem}
Consider the following multiple quadratic knapsack problem (MQKP) with $d$ knapsacks and $n$ items:
\begin{equation}\label{eq-MQKP}
\max\left\{x^\top Q\,x + c^\top x
\mid
Ax\leq b,\;
x\in\{0,1\}^n
\right\},
\end{equation}
where \(Q\in\mathbb{R}^{n\times n}\) is a symmetric profit matrix, \(c\in\mathbb{R}^n\) is a linear profit vector, \(A\in\mathbb{R}^{d\times n}\) has entries \(A_{ij}\) denoting the weight of item \(j\) in knapsack \(i\), and \(b\in\mathbb{R}^d\) is the capacity vector.  To convert the problem into the standard form \eqref{eq-POP}, we follow the strategy described in Subsection~\ref{subsec-ineq} by introducing nonnegative slack variables \(s\in\mathbb{R}^d_+\) for each capacity constraint, leading to the equivalent reformulation:
\begin{equation}\label{eq-MQKP-slack}
\max\Big\{x^\top Q\,x + c^\top x
\mid
Ax+s=b,\; 
x\in\{0,1\}^n,\; (x,s)\in\R^{n+d}_+
\Big\}.\tag{MQKP}
\end{equation}
We generate test instances using the method described by Gallo et al.~\cite{gallo1980quadratic}. For each problem size \(n\in\{10,20,30,40\}\) and knapsack number $d=\{3,6\}$, we let the probability of a nonzero quadratic interaction be \(p=0.25\) and for each \(i<j\) set
\[
Q_{jj}=0,\quad Q_{ij}=Q_{ji}=\begin{cases}
0,&\text{with probability }1-p,\\
\text{Uniform}\{1,\dots,100\},&\text{with probability }p.
\end{cases}
\]
We then draw each linear coefficient \(c_j\) uniformly from \([1,100]\) and each weight \(A_{ij}\) uniformly from \([1,50]\).  Finally, we set the capacity of knapsack \(i\) to
$b_i \;=\;\lceil 0.1\sum_{j=1}^n A_{ij}\rceil$.
Since the first-order Poly-SDP-RLT relaxation of \eqref{eq-MQKP-slack} is not tight, we solve the second-order relaxation using SDPNAL+ and RiNNAL-POP.  Numerical results are reported in Table~\ref{table:MQKP2}.

{\small
\begin{longtable}{@{}l l l r l r r@{}}
\caption{Computational results for MQKP problems (relaxation order 
$\tau=2$).}\label{table:MQKP2} \\
\toprule
{(n, N, d)}, m & \multicolumn{1}{l}{Algorithm} & \multicolumn{1}{l}{Iteration} & \multicolumn{1}{c}{Rank} & \multicolumn{1}{c}{$R_{\max}$} & \multicolumn{1}{c}{Objective} & {Time} \\
\midrule
\endfirsthead
\toprule
{(n, N, d)}, m & \multicolumn{1}{l}{Algorithm} & \multicolumn{1}{l}{Iteration} & \multicolumn{1}{c}{Rank} & \multicolumn{1}{c}{$R_{\max}$} & \multicolumn{1}{c}{Objective} & {Time} \\
\midrule
\endhead
\multirow[t]{2}{*}[0pt]{\parbox[t]{2cm}{(20,300,3)\\m=55201}} & RPOP & 28, 5800, 28 & 1 & 8.74e-08 & -2.1400018e+02 & 7.6 \\
 & SDPNAL+ & 18, 86, 500 & 1 & 1.25e-07 & -2.1400000e+02 & 88.3 \\
\addlinespace
\multirow[t]{2}{*}[0pt]{\parbox[t]{2cm}{(20,378,6)\\m=113023}} & RPOP & 26, 5400, 26 & 1 & 7.80e-07 & -1.6800001e+02 & 10.0 \\
 & SDPNAL+ & 12, 42, 350 & 1 & 8.63e-08 & -1.6800000e+02 & 101.3 \\
\addlinespace
\multirow[t]{2}{*}[0pt]{\parbox[t]{2cm}{(30,595,3)\\m=189806}} & RPOP & 34, 7000, 34 & 1 & 3.19e-07 & -7.8400063e+02 & 29.0 \\
 & SDPNAL+ & 34, 177, 2100 & 1 & 1.06e-07 & -7.8400003e+02 & 900.6 \\
\addlinespace
\multirow[t]{2}{*}[0pt]{\parbox[t]{2cm}{(30,703,6)\\m=333223}} & RPOP & 90, 18200, 90 & 2 & 4.54e-07 & -3.9600026e+02 & 114.8 \\
 & SDPNAL+ & 54, 250, 3700 & 1 & 1.28e-08 & -3.9600000e+02 & 2958.7 \\
\addlinespace
\multirow[t]{2}{*}[0pt]{\parbox[t]{2cm}{(40,990,3)\\m=482461}} & RPOP & 100, 20200, 100 & 4 & 8.16e-07 & -1.1970070e+03 & 223.4 \\
 & SDPNAL+ & 46, 252, 2900 & 12 & 5.96e-05$^{\dagger}$ & -1.1969960e+03 & 3639.2 \\
\addlinespace
\multirow[t]{2}{*}[0pt]{\parbox[t]{2cm}{(40,1128,6)\\m=769673}} & RPOP & 848, 169800, 848 & 237 & 2.43e-04$^{\dagger}$ & -9.0042934e+02 & 3603.1 \\
 & SDPNAL+ & 25, 113, 1300 & 469 & 1.45e-03$^{\dagger}$ & -9.2321049e+02 & 3708.0 \\
\addlinespace
\bottomrule
\end{longtable}
}

As shown in Table~\ref{table:MQKP2}, RPOP consistently solves all instances to the required accuracy, recovering rank-one or nearly rank-one solutions except the last instance. In contrast, SDPNAL+ experiences significant difficulty on the larger instances. For $n \leq 30$, both solvers achieve the required accuracy, but RPOP is 10–30 times faster. On the more challenging instances with $n=40$, the performance gap widens substantially. For example, for the instance with $d=3$, RPOP converges within 4 minutes, while SDPNAL+ terminates with a high-rank solution and fails to meet the required accuracy within the 1-hour time limit.

\subsection{Ball‐constrained quartic minimization}
Consider the following quartic minimization problem over the $(n-1)$‐dimensional Euclidean unit ball:
\begin{equation}\label{eq-QB}
\min\Big\{\langle c,[x]_4\rangle \mid \|x\|\leq 1,\; x\in\R^{n-1}\Big\},
\end{equation}
where $[x]_4$ is the vector of monomials in $x$ of degree up to 4, and $c \in \R^\ell$ is a given coefficient vector with $\ell:={\binom{n+3}{4}}$. To convert the problem into the standard form \eqref{eq-POP}, we follow the strategy described in Subsection~\ref{subsec-ineq} by introducing a squared slack variable $s\in\R$ and equivalently reformulate~\eqref{eq-QB} as a sphere-constrained problem:
\begin{equation}\label{eq-QB-slack}
\min\Big\{\langle c,[x]_4\rangle \mid \|x\|^2+s^2= 1,\; (x,s)\in\R^n\Big\}.\tag{BQM}
\end{equation}
For each $n\in\{20,40,60,80\}$, we generate random instances by sampling each coefficient $c_i\sim\mathcal N(0,1)$.  
For this special problem \eqref{eq-QB-slack} with a constraint over the unit sphere, the solver ManiSDP \cite{wang2023solving} can be used to solve the relaxation problems.
We solve the resulting second-order Poly-SDP-RLT relaxation of \eqref{eq-QB-slack} using SDPNAL+, ManiSDP, and RiNNAL-POP, and summarize the numerical results in Table~\ref{table:QB2}.

{\small
\begin{longtable}{@{}l l l r l r r@{}}
\caption{Computational results for BQM problems (relaxation 
order $\tau=2$).}\label{table:QB2} \\
\toprule
{(n, N), m} & \multicolumn{1}{l}{Algorithm} & \multicolumn{1}{l}{Iteration} & \multicolumn{1}{c}{Rank} & \multicolumn{1}{c}{$R_{\max}$} & \multicolumn{1}{c}{Objective} & {Time} \\
\midrule
\endfirsthead
\toprule
{(n, N), m} & \multicolumn{1}{l}{Algorithm} & \multicolumn{1}{l}{Iteration} & \multicolumn{1}{c}{Rank} & \multicolumn{1}{c}{$R_{\max}$} & \multicolumn{1}{c}{Objective} & {Time} \\
\midrule
\endhead
\multirow[t]{3}{*}[0pt]{\parbox[t]{2cm}{(20,231)\\m=16402}} & RPOP & 25, 520, 25 & 1 & 2.65e-07 & -1.2346539e+01 & 0.6 \\
 & SDPNAL+ & 11, 23, 300 & 1 & 3.37e-07 & -1.2346534e+01 & 3.5 \\
 & ManiSDP & 96 & 1 & 1.96e-07 & -1.2346549e+01 & 6.8 \\
\addlinespace
\multirow[t]{3}{*}[0pt]{\parbox[t]{2cm}{(40,861)\\m=236202}} & RPOP & 21, 440, 21 & 1 & 5.18e-07 & -2.0874267e+01 & 4.4 \\
 & SDPNAL+ & 6, 27, 623 & 2 & 9.80e-07 & -2.0874046e+01 & 72.3 \\
 & ManiSDP & 180 & 10 & 1.35e-02$^{\dagger}$ & -2.1770028e+01 & 93.2 \\
\addlinespace
\multirow[t]{3}{*}[0pt]{\parbox[t]{2cm}{(60,1891)\\m=1155402}} & RPOP & 44, 900, 44 & 1 & 7.76e-07 & -2.2240284e+01 & 67.1 \\
 & SDPNAL+ & 8, 30, 599 & 2 & 9.59e-07 & -2.2239502e+01 & 491.2 \\
 & ManiSDP & - & - & - & - & - \\
\addlinespace
\multirow[t]{3}{*}[0pt]{\parbox[t]{2cm}{(80,3321)\\m=3590002}} & RPOP & 51, 1040, 51 & 1 & 2.04e-07 & -2.7585001e+01 & 372.1 \\
 & SDPNAL+ & 8, 30, 588 & 2 & 9.54e-07 & -2.7583874e+01 & 3167.5 \\
 & ManiSDP & - & - & - & - & - \\
\addlinespace
\bottomrule
\end{longtable}
}

As shown in Table~\ref{table:QB2}, both RPOP and SDPNAL+ successfully solve all test instances to the required accuracy, consistently recovering rank-one or low-rank solutions. In contrast, ManiSDP fails to solve problems with $n\geq40$, due to early termination from its stagnation-based stopping criterion.
At $n=40$, RPOP is over 15 times faster than SDPNAL+, and more than 20 times faster than ManiSDP. For larger instances ($n=60, 80$), RPOP continues to outperform SDPNAL+ with speedups of approximately 7–9 times, while ManiSDP fails to converge. These results demonstrate the scalability and numerical stability of RPOP in solving Poly-SDP-RLT relaxations of dense quartic problems over high-dimensional unit ball.

\subsection{Kurtosis‐minimization portfolio problem}
Consider the kurtosis‐minimization portfolio model under fixed expectation and risk levels:
\begin{equation}\label{eq-KM}
\min\Bigl\{\mathbb{E}_\xi\bigl[(x^\top (\xi - \mu))^4\bigr]\mid\mu^\top x = \mu_0,\;x^\top\Sigma x = \sigma_0^2,\;e^\top x = 1,\; x\ge 0\Bigr\},\tag{KMP}
\end{equation} 
where $\xi\in\R^n$ denotes the random asset‐return vector, $\mu=\mathbb{E}[\xi]\in\R^n$ its mean vector, $\Sigma=\mathrm{Cov}(\xi)\in\R^{n\times n}$ the covariance matrix, $e\in\R^n$ the all‐ones vector, and $x\in\R^n$ the decision variable representing long-only portfolio weights. The parameters \(\mu_0\) and \(\sigma_0^2\) denote the target expected return and variance, respectively. 

Problem~\eqref{eq-KM} corresponds to P3 in \cite{mhiri2010international}, and we generate synthetic data following the setup in \cite{hu2022globally}. For each \(n\in\{20,40,60,80\}\), we draw \(p=255\) independent asset-price samples $\xi_i\in\R^n$ uniformly from $(0,10)^n$. Letting \(\hat\mu\) and \(\hat\Sigma\) denote the empirical mean and covariance of \(\{\xi_i\}_{i=1}^p\), we set
\[
\mu_0 = \hat\mu^\top\Bigl(\frac{e}{n}\Bigl), 
\qquad
\sigma_0^2 = \Bigl(\frac{e}{n}\Bigr)^\top \hat\Sigma \Bigl(\frac{e}{n}\Bigr),
\]
so that the targets match those of the equally‐weighted portfolio \(x = e/n\). We solve the second-order Poly-SDP-RLT relaxation using RPOP and SDPNAL+. Table~\ref{table:KM2} summarizes the results.

{\small
\begin{longtable}{@{}l l l r l r r@{}}
\caption{Computational results for KM problems (relaxation order
$\tau=2$).}\label{table:KM2} \\
\toprule
{(n, N), m} & \multicolumn{1}{l}{Algorithm} & \multicolumn{1}{l}{Iteration} & \multicolumn{1}{c}{Rank} & \multicolumn{1}{c}{$R_{\max}$} & \multicolumn{1}{c}{Objective} & {Time} \\
\midrule
\endfirsthead
\toprule
{(n, N), m} & \multicolumn{1}{l}{Algorithm} & \multicolumn{1}{l}{Iteration} & \multicolumn{1}{c}{Rank} & \multicolumn{1}{c}{$R_{\max}$} & \multicolumn{1}{c}{Objective} & {Time} \\
\midrule
\endhead
\multirow[t]{2}{*}[0pt]{\parbox[t]{2cm}{(20,231)\\m=26104}} & RPOP & 14, 750, 14 & 1 & 4.81e-07 & 5.3904621e-01 & 0.7 \\
 & SDPNAL+ & 85, 405, 3337 & 1 & 9.66e-07 & 5.4569096e-01 & 194.5 \\
\addlinespace
\multirow[t]{2}{*}[0pt]{\parbox[t]{2cm}{(40,861)\\m=306804}} & RPOP & 41, 2100, 41 & 5 & 6.76e-07 & 8.0918582e-02 & 18.9 \\
 & SDPNAL+ & 77, 634, 3650 & 39 & 6.24e-06$^{\dagger}$ & 9.0307077e-02 & 3602.7 \\
\addlinespace
\multirow[t]{2}{*}[0pt]{\parbox[t]{2cm}{(60,1891)\\m=1386104}} & RPOP & 172, 8650, 172 & 56 & 9.98e-07 & 3.0082122e-02 & 464.3 \\
 & SDPNAL+ & 20, 67, 650 & 112 & 1.95e-04$^{\dagger}$ & 7.2863598e-02 & 3626.3 \\
\addlinespace
\multirow[t]{1}{*}[0pt]{\parbox[t]{2cm}{(80,3321)\\m=4128004}} & RPOP & 139, 7000, 139 & 74 & 9.97e-07 & 1.3900566e-02 & 1831.2 \\
 & SDPNAL+ & - & - & - & - & - \\
\addlinespace
\bottomrule
\end{longtable}
}

As shown in Table~\ref{table:KM2}, RPOP successfully solves all instances to the required accuracy, while SDPNAL+ fails to converge within the 1-hour time limit for instances with $n \geq 40$. At $n = 20$, RPOP is over 270 times faster than SDPNAL+. As the problem size increases, the performance gap widens significantly: for $n = 40$, RPOP converges within 20 seconds, while SDPNAL+ fails to converge within the 1-hour time limit; at $n = 60$, RPOP solves the problem in 8 minutes, whereas SDPNAL+ terminates with a large KKT residual. For $n = 80$, only RPOP is able to solve the instance, as SDPNAL+ encounters memory issues due to over 4 million equality constraints and more than 11 million nonnegativity constraints. The high efficiency of RPOP is partly attributed to the highly structured nature of the relaxation, which introduces numerous facial equality constraints that are effectively exploited in the design of our algorithm.

\subsection{Copositivity of symmetric tensors}
Let \(\B\in\operatorname{Sym}\left(\otimes^{\rmt}\R^{n}\right)\) be a symmetric tensor of order $\rmt$. The tensor \(\B\)  is said to be copositive if
\[
\bigl\langle
\B,\;\otimes^\rmt x
\bigr\rangle
\;\ge\;0
\quad\text{for all }x\in\R^{n}_{+}.
\]
When \(\rmt=2\), this definition reduces to the classical notion of matrix copositivity. Testing tensor copositivity for general \(\rmt\) is co‐NP‐hard \cite{dickinson2014computational,murty1985some}, and can be reformulated as the polynomial optimization problem
\begin{equation}\label{eq-CST}
\min\left\{ \langle \B, \otimes^\rmt x\rangle \mid  e^\top x=1,\; x\geq 0\right\}.\tag{CST}
\end{equation}
The tensor $\B$ is copositive if and only if the optimal value of \eqref{eq-CST} is nonnegative.
We consider two types of test instances constructed according to the method described in \cite{qi2013symmetric}:
\begin{enumerate}
    \item \textbf{Random symmetric tensors}. Each entry of \(\B\) is sampled independently from the uniform distribution on \([-1,1]\), and full symmetry is enforced by averaging over index permutations.
    \item \textbf{Guaranteed-copositive tensors}. 
    A sufficient condition for copositivity is that
    \[
    \B_{i\,i\cdots i}\;\ge\;
    -\sum\bigl\{\,\B_{i\,i_2\cdots i_\rmt} :
    (i,i_2,\dots,i_\rmt)\neq(i,i,\dots,i)\text{ and }\B_{i\,i_2\cdots i_\rmt}<0\bigr\}
    \quad\forall\,i=1,\dots,n.
    \]
    To construct such tensors, we first generate a random symmetric tensor as above. Then, for each \(i=1,\dots,n\), we set
    \[
    \B_{ii\cdots i}
    \;=\;
    10^{-6}
    \;-\;
    \sum\bigl\{\,\B_{ii_2\cdots i_\rmt} :
    (i,i_2,\dots,i_\rmt)\neq(i,i,\dots,i)\text{ and }\B_{i\,i_2\cdots i_\rmt}<0\bigr\},
    \]
    which ensures that the copositivity condition is satisfied and $\B$ is copositive by construction.
\end{enumerate}
For matrix instances ($\rmt=2$), the first‐order Poly‐SDP‐RLT relaxation is already tight. Thus, we test the first-order relaxation with $n\in\{1000,2000,4000,6000\}$. For the fourth‐order tensor instances ($\rmt=4$), the relaxation order must be at least 2 to capture the degree of the polynomial. Accordingly, we apply the second‐order Poly‐SDP‐RLT relaxation with $n\in\{20,40,60\}$. The numerical results are summarized in Tables~\ref{table:CST2} and~\ref{table:CST4}. In these tables, DT denotes the data type.

{\small
\begin{longtable}{@{}l l l l r l r r r@{}}
\caption{Computational results for CST problems with $\rmt=2$ (relaxation order $\tau = 1$).}\label{table:CST2} \\
\toprule
DT & {(n, N), m} & \multicolumn{1}{l}{Algorithm} & \multicolumn{1}{l}{Iteration} & \multicolumn{1}{c}{Rank} & \multicolumn{1}{c}{$R_{\max}$} & \multicolumn{1}{c}{Objective} & {Time} \\
\midrule
\endfirsthead
\toprule
DT & {(n, N), m} & \multicolumn{1}{l}{Algorithm} & \multicolumn{1}{l}{Iteration} & \multicolumn{1}{c}{Rank} & \multicolumn{1}{c}{$R_{\max}$} & \multicolumn{1}{c}{Objective} & {Time} \\
\midrule
\endhead
\multirow[t]{2}{*}[0pt]{1} & \multirow[t]{2}{*}[0pt]{\parbox[t]{2cm}{(1000,1001)\\m=1002}} & RPOP & 116, 5843, 116 & 1 & 7.96e-07 & -9.9221196e-01 & 38.2 \\
 & & SDPNAL+ & 227, 247, 5330 & 2 & 1.66e-06$^{\dagger}$ & -9.9184447e-01 & 702.1 \\
\addlinespace
\multirow[t]{2}{*}[0pt]{1} & \multirow[t]{2}{*}[0pt]{\parbox[t]{2cm}{(2000,2001)\\m=2002}} & RPOP & 136, 6850, 136 & 4 & 8.86e-07 & -9.9850085e-01 & 300.2 \\
 & & SDPNAL+ & 130, 201, 3307 & 19 & 1.86e-05$^{\dagger}$ & -9.9420320e-01 & 3601.3 \\
\addlinespace
\multirow[t]{2}{*}[0pt]{2} & \multirow[t]{2}{*}[0pt]{\parbox[t]{2cm}{(1000,1001)\\m=1002}} & RPOP & 2, 80, 2 & 1 & 1.74e-14 & 1.6615831e-01 & 0.9 \\
 & & SDPNAL+ & 45, 45, 2307 & 1 & 9.10e-07 & 1.6799439e-01 & 221.9 \\
\addlinespace
\multirow[t]{2}{*}[0pt]{2} & \multirow[t]{2}{*}[0pt]{\parbox[t]{2cm}{(2000,2001)\\m=2002}} & RPOP & 2, 80, 2 & 1 & 5.88e-15 & 1.6655448e-01 & 4.2 \\
 & & SDPNAL+ & 105, 108, 2723 & 1 & 8.54e-07 & 1.7012510e-01 & 2107.7 \\
\addlinespace
\multirow[t]{2}{*}[0pt]{2} & \multirow[t]{2}{*}[0pt]{\parbox[t]{2cm}{(4000,4001)\\m=4002}} & RPOP & 2, 84, 2 & 1 & 2.03e-14 & 1.6643175e-01 & 30.9 \\
 & & SDPNAL+ & 20, 20, 458 & 797 & 1.21e-05$^{\dagger}$ & 2.7370011e+00 & 3605.1 \\
\addlinespace
\multirow[t]{2}{*}[0pt]{2} & \multirow[t]{2}{*}[0pt]{\parbox[t]{2cm}{(6000,6001)\\m=6002}} & RPOP & 1, 77, 1 & 1 & 7.17e-07 & 1.6667974e-01 & 68.1 \\
 & & SDPNAL+ & 0, 0, 153 & 1 & 5.39e-06$^{\dagger}$ & 5.0527070e-01 & 3613.0 \\
\addlinespace
\bottomrule
\end{longtable}
}

As shown in Table~\ref{table:CST2}, RPOP successfully solves all matrix copositivity instances using the first-order Poly-SDP-RLT relaxation, whereas SDPNAL+ is unable to solve several large-scale or difficult instances within the 1-hour time limit, particularly for data type 2. For these problems, RPOP achieves up to 500 times speedups and consistently delivers low-rank solutions.

{\small
\begin{longtable}{@{}l l l l r l r r r@{}}
\caption{Computational results for CST problems with $\rmt=4$ (relaxation order $\tau=2$).}\label{table:CST4} \\
\toprule
DT & {(n, N), m} & \multicolumn{1}{l}{Algorithm} & \multicolumn{1}{l}{Iteration} & \multicolumn{1}{c}{Rank} & \multicolumn{1}{c}{$R_{\max}$} & \multicolumn{1}{c}{Objective} & {Time} \\
\midrule
\endfirsthead
\toprule
DT & {(n, N), m} & \multicolumn{1}{l}{Algorithm} & \multicolumn{1}{l}{Iteration} & \multicolumn{1}{c}{Rank} & \multicolumn{1}{c}{$R_{\max}$} & \multicolumn{1}{c}{Objective} & {Time} \\
\midrule
\endhead
\multirow[t]{2}{*}[0pt]{1} & \multirow[t]{2}{*}[0pt]{\parbox[t]{2cm}{(20,231)\\m=21022}} & RPOP & 14, 750, 14 & 2 & 6.42e-07 & -9.4311726e-01 & 0.7 \\
 & & SDPNAL+ & 15, 17, 830 & 1 & 8.50e-07 & -9.4292047e-01 & 9.3 \\
\addlinespace
\multirow[t]{2}{*}[0pt]{1} & \multirow[t]{2}{*}[0pt]{\parbox[t]{2cm}{(40,861)\\m=270642}} & RPOP & 30, 1549, 30 & 1 & 1.60e-07 & -9.6583382e-01 & 17.6 \\
 & & SDPNAL+ & 90, 116, 2694 & 3 & 8.68e-07 & -9.6482563e-01 & 399.8 \\
\addlinespace
\multirow[t]{2}{*}[0pt]{1} & \multirow[t]{2}{*}[0pt]{\parbox[t]{2cm}{(60,1891)\\m=1268862}} & RPOP & 54, 2733, 54 & 1 & 3.75e-07 & -9.9430793e-01 & 217.5 \\
 & & SDPNAL+ & 50, 99, 3226 & 52 & 1.80e-05$^{\dagger}$ & -9.9445335e-01 & 3601.1 \\
\addlinespace
\multirow[t]{2}{*}[0pt]{2} & \multirow[t]{2}{*}[0pt]{\parbox[t]{2cm}{(20,231)\\m=21022}} & RPOP & 18, 950, 18 & 1 & 3.35e-07 & 5.3540264e-02 & 0.9 \\
 & & SDPNAL+ & 83, 599, 3477 & 1 & 1.00e-06 & 5.3809530e-02 & 156.3 \\
\addlinespace
\multirow[t]{2}{*}[0pt]{2} & \multirow[t]{2}{*}[0pt]{\parbox[t]{2cm}{(40,861)\\m=270642}} & RPOP & 19, 1000, 19 & 1 & 7.13e-07 & 5.0214195e-02 & 9.4 \\
 & & SDPNAL+ & 31, 172, 1001 & 1 & 8.63e-07 & 5.0348184e-02 & 514.5 \\
\addlinespace
\multirow[t]{2}{*}[0pt]{2} & \multirow[t]{2}{*}[0pt]{\parbox[t]{2cm}{(60,1891)\\m=1268862}} & RPOP & 21, 1100, 21 & 1 & 5.27e-07 & 4.8896302e-02 & 71.2 \\
 & & SDPNAL+ & 18, 103, 322 & 1213 & 9.89e-07 & 1.9684112e+00 & 2185.9 \\
\addlinespace
\bottomrule
\end{longtable}
}

In the tensor setting (Table~\ref{table:CST4}), RPOP again demonstrates clear superiority. For data type 1, it is typically 10–20 times faster than SDPNAL+, and for data type 2, the speedup reaches 30–150 times. Notably, even when both solvers satisfy the stopping criteria, SDPNAL+ may return objective values with noticeable deviation (e.g., $n=60$) due to looser complementarity tolerances. With a smaller stopping tolerance, SDPNAL+ would converge to the same value as RPOP. These results highlight RPOP’s efficiency and robustness on copositivity verification tasks for both matrix and higher-order tensor cases.

\subsection{Nonnegative tensor factorization}

We consider the best nonnegative rank-one approximation of a given tensor. Let
\begin{equation*}
    \B\in \operatorname{Sym}\left(\otimes^{\alpha_1}\R^{n_1}\right)
\;\otimes\;\cdots\;\otimes\;
\operatorname{Sym}\left(\otimes^{\alpha_p}\R^{n_p}\right)
\end{equation*}
be a partially symmetric tensor of order \(\sum_{i=1}^p\alpha_i\). The best nonnegative rank-one tensor approximation problem is
\begin{equation}\label{eq-NTF}
\min\left\{ \|\B-\lambda x^{\otimes \alpha}\|^2 \mid \ \|x^{(i)}\|=1,\; \lambda\geq 0,\; x^{(i)}\geq 0,\; \forall i\in[p]\right\},
\end{equation}
where each \(x^{(i)}\in\R_+^{n_i}\) has unit Euclidean norm, \(\lambda\ge0\) is a scaling factor, and
\[
(x^{(i)})^{\otimes\alpha_i}
:=\underbrace{x^{(i)}\otimes x^{(i)}\otimes\cdots\otimes x^{(i)}}_{\alpha_i\ \text{copies}},
\]
so that 
\(\;x^{\otimes\alpha}=(x^{(1)})^{\otimes\alpha_1}\otimes\cdots\otimes(x^{(p)})^{\otimes\alpha_p}\)
is the resulting rank-one tensor.  Here \(\|\cdot\|\) denotes the Frobenius norm, i.e.\ 
\(\|\cT\|^2=\sum_{i_1,\dots,i_p}\cT_{i_1\cdots i_p}^2\).
Eliminating the scalar \(\lambda\) yields the equivalent homogeneous formulation \cite{hu2019best}:
\begin{equation}\label{eq-NTF-reduced}
\min\left\{ \langle -\B, x^{\otimes \alpha}\rangle \mid \ \|x^{(i)}\|=1,\; x^{(i)}\geq 0,\; \forall i\in[p]\right\}.\tag{NTF}
\end{equation}
We test two classes of tensors that yield homogeneous polynomial objectives of degree \(\rmt\) in the entries of \(x^{(i)}\):
\begin{enumerate}
   \item \textbf{General Tensors (NTF).}  Set \(p=\rmt\), \(\alpha_i=1\), and \(n_i=n\) so that \(\cA\in\otimes^\rmt\R^n\).  The entries, given in \cite[Example 3.16]{nie2014semidefinite}, are defined as
  \[
    \B_{i_1\cdots i_\rmt}
      =\sum_{j=1}^\rmt (-1)^{j+1}\,j\,e^{-\,i_j}.
  \]
  \item \textbf{Symmetric Tensors (NSTF).} Set \(p=1\), \(\alpha_1=\rmt\), \(n_1=n\), so that \(\B\in\operatorname{Sym}(\otimes^\rmt\R^n)\). We test three types of symmetric entries:
  \begin{align}
    \B_{i_1\cdots i_\rmt}
      &=\sum_{j=1}^\rmt \frac{(-1)^{i_j}}{i_j}, 
    \tag{NSTF-1}\label{eq:DT1}\\
    \B_{i_1\cdots i_\rmt}
      &=\sum_{j=1}^\rmt \arctan\left(\frac{(-1)^{i_j}i_j}{n}\right),
    \tag{NSTF-2}\label{eq:DT2}\\
    \B_{i_1\cdots i_\rmt}
      &=\sum_{j=1}^\rmt (-1)^{i_j}\,\log(i_j).
    \tag{NSTF-3}\label{eq:DT3}
  \end{align}
  These examples are taken from \cite[Examples 3.5–3.7]{nie2014semidefinite}.
\end{enumerate}

We test tensors with order \(\rmt\in\{3,4\}\) and apply the second-order Poly‐SDP‐RLT relaxation to compute lower bounds. Numerical results for the NTF instances are reported in Tables~\ref{table:NTF-3} and~\ref{table:NTF-4}, while results for the NSTF instances are shown in Tables~\ref{table:NSTF-3} and~\ref{table:NSTF-4}.

{\small
\begin{longtable}{@{}l l l r l r r@{}}
\caption{Computational results for NTF problems with $\rmt=3$ (relaxation order $\tau=2$).}\label{table:NTF-3} \\
\toprule
{(n, N), m} & \multicolumn{1}{l}{Algorithm} & \multicolumn{1}{l}{Iteration} & \multicolumn{1}{c}{Rank} & \multicolumn{1}{c}{$R_{\max}$} & \multicolumn{1}{c}{Objective} & {Time} \\
\midrule
\endfirsthead
\toprule
{(n, N), m} & \multicolumn{1}{l}{Algorithm} & \multicolumn{1}{l}{Iteration} & \multicolumn{1}{c}{Rank} & \multicolumn{1}{c}{$R_{\max}$} & \multicolumn{1}{c}{Objective} & {Time} \\
\midrule
\endhead
\multirow[t]{2}{*}[0pt]{\parbox[t]{2cm}{(10,496)\\m=78369}} & RPOP & 25, 1300, 25 & 4 & 6.70e-07 & -1.1611106e+01 & 3.4 \\
 & SDPNAL+ & 2, 2, 353 & 4 & 8.65e-07 & -1.1611100e+01 & 9.1 \\
\addlinespace
\multirow[t]{2}{*}[0pt]{\parbox[t]{2cm}{(15,1081)\\m=376189}} & RPOP & 28, 1450, 28 & 4 & 7.97e-08 & -1.7610440e+01 & 20.8 \\
 & SDPNAL+ & 11, 11, 694 & 7 & 4.41e-07 & -1.7610440e+01 & 87.4 \\
\addlinespace
\multirow[t]{2}{*}[0pt]{\parbox[t]{2cm}{(20,1891)\\m=1159184}} & RPOP & 27, 1400, 27 & 4 & 5.39e-07 & -2.3579847e+01 & 74.4 \\
 & SDPNAL+ & 0, 0, 407 & 6 & 9.36e-07 & -2.3579808e+01 & 176.6 \\
\addlinespace
\multirow[t]{2}{*}[0pt]{\parbox[t]{2cm}{(25,2926)\\m=2788479}} & RPOP & 31, 1600, 31 & 4 & 1.30e-07 & -2.9535838e+01 & 362.0 \\
 & SDPNAL+ & 11, 11, 651 & 9 & 9.33e-07 & -2.9535838e+01 & 1472.6 \\
\addlinespace
\bottomrule
\end{longtable}
}

As shown in Table~\ref{table:NTF-3}, both RPOP and SDPNAL+ successfully solve all second-order relaxation problems to the required accuracy. However, RPOP is consistently 3–6 times faster than SDPNAL+.

{\small
\begin{longtable}{@{}l l l r l r r@{}}
\caption{Computational results for NTF problems with $\rmt=4$ (relaxation order $\tau=2$).}\label{table:NTF-4} \\
\toprule
{(n, N), m} & \multicolumn{1}{l}{Algorithm} & \multicolumn{1}{l}{Iteration} & \multicolumn{1}{c}{Rank} & \multicolumn{1}{c}{$R_{\max}$} & \multicolumn{1}{c}{Objective} & {Time} \\
\midrule
\endfirsthead
\toprule
{(n, N), m} & \multicolumn{1}{l}{Algorithm} & \multicolumn{1}{l}{Iteration} & \multicolumn{1}{c}{Rank} & \multicolumn{1}{c}{$R_{\max}$} & \multicolumn{1}{c}{Objective} & {Time} \\
\midrule
\endhead
\multirow[t]{2}{*}[0pt]{\parbox[t]{2cm}{(10,861)\\m=238785}} & RPOP & 28, 1450, 28 & 8 & 6.40e-07 & -3.1637169e+01 & 12.6 \\
 & SDPNAL+ & 11, 11, 644 & 8 & 9.78e-07 & -3.1637174e+01 & 50.3 \\
\addlinespace
\multirow[t]{2}{*}[0pt]{\parbox[t]{2cm}{(15,1891)\\m=1161075}} & RPOP & 32, 1650, 32 & 8 & 5.62e-07 & -6.1949211e+01 & 88.6 \\
 & SDPNAL+ & 13, 13, 1298 & 17 & 1.58e-06$^{\dagger}$ & -6.1949379e+01 & 586.2 \\
\addlinespace
\multirow[t]{2}{*}[0pt]{\parbox[t]{2cm}{(20,3321)\\m=3599965}} & RPOP & 42, 2150, 42 & 8 & 4.28e-07 & -9.8202747e+01 & 632.5 \\
 & SDPNAL+ & 0, 0, 618 & 12 & 9.55e-07 & -9.8202557e+01 & 1566.0 \\
\addlinespace
\bottomrule
\end{longtable}
}

As shown in Table~\ref{table:NTF-4}, both RPOP and SDPNAL+ successfully solve all the instances. However, RPOP consistently achieves a 4–7 times speedup over SDPNAL+.

{\small
\begin{longtable}{@{}l l l l r r r r r@{}}
\caption{Computational results for NSTF problems with $\rmt=3$
(relaxation order $\tau=2$).}\label{table:NSTF-3} \\
\toprule
NSTF & {(n, N), m} & \multicolumn{1}{l}{Algorithm} & \multicolumn{1}{l}{Iteration} & \multicolumn{1}{c}{Rank} & \multicolumn{1}{c}{$R_{\max}$} & \multicolumn{1}{c}{Objective} & {Time} \\
\midrule
\endfirsthead
\toprule
NSTF & {(n, N), m} & \multicolumn{1}{l}{Algorithm} & \multicolumn{1}{l}{Iteration} & \multicolumn{1}{c}{Rank} & \multicolumn{1}{c}{$R_{\max}$} & \multicolumn{1}{c}{Objective} & {Time} \\
\midrule
\endhead
\multirow[t]{2}{*}[0pt]{1} & \multirow[t]{2}{*}[0pt]{\parbox[t]{2cm}{(40,861)\\m=236202}} & RPOP & 19, 1000, 19 & 1 & 2.68e-07 & -3.8016921e+01 & 9.0 \\
 & & SDPNAL+ & 0, 0, 331 & 4 & 9.75e-07 & -3.8016959e+01 & 25.0 \\
\addlinespace
\multirow[t]{2}{*}[0pt]{1} & \multirow[t]{2}{*}[0pt]{\parbox[t]{2cm}{(60,1891)\\m=1155402}} & RPOP & 19, 1000, 19 & 1 & 8.35e-07 & -5.6213703e+01 & 54.7 \\
 & & SDPNAL+ & 1, 1, 479 & 1 & 5.47e-07 & -5.6214490e+01 & 207.0 \\
\addlinespace
\multirow[t]{2}{*}[0pt]{1} & \multirow[t]{2}{*}[0pt]{\parbox[t]{2cm}{(80,3321)\\m=3590002}} & RPOP & 29, 1500, 29 & 2 & 9.83e-07 & -7.4060045e+01 & 455.1 \\
 & & SDPNAL+ & 15, 20, 1000 & 1 & 9.16e-07 & -7.4060013e+01 & 3403.3 \\
\addlinespace
\multirow[t]{2}{*}[0pt]{2} & \multirow[t]{2}{*}[0pt]{\parbox[t]{2cm}{(40,861)\\m=236202}} & RPOP & 48, 2450, 48 & 2 & 3.94e-07 & -1.5094742e+02 & 20.9 \\
 & & SDPNAL+ & 22, 25, 650 & 1 & 9.58e-07 & -1.5094484e+02 & 76.8 \\
\addlinespace
\multirow[t]{2}{*}[0pt]{2} & \multirow[t]{2}{*}[0pt]{\parbox[t]{2cm}{(60,1891)\\m=1155402}} & RPOP & 66, 3350, 66 & 1 & 8.90e-07 & -2.7343951e+02 & 179.8 \\
 & & SDPNAL+ & 23, 29, 650 & 1 & 9.82e-07 & -2.7342993e+02 & 442.9 \\
\addlinespace
\multirow[t]{2}{*}[0pt]{2} & \multirow[t]{2}{*}[0pt]{\parbox[t]{2cm}{(80,3321)\\m=3590002}} & RPOP & 61, 3100, 61 & 1 & 7.98e-07 & -4.1800259e+02 & 931.1 \\
 & & SDPNAL+ & 25, 36, 550 & 1 & 8.59e-07 & -4.1796967e+02 & 2431.9 \\
\addlinespace
\multirow[t]{2}{*}[0pt]{3} & \multirow[t]{2}{*}[0pt]{\parbox[t]{2cm}{(40,861)\\m=236202}} & RPOP & 46, 2350, 46 & 1 & 2.56e-08 & -8.6565039e+02 & 20.2 \\
 & & SDPNAL+ & 21, 21, 650 & 1 & 5.78e-07 & -8.6564524e+02 & 68.8 \\
\addlinespace
\multirow[t]{2}{*}[0pt]{3} & \multirow[t]{2}{*}[0pt]{\parbox[t]{2cm}{(60,1891)\\m=1155402}} & RPOP & 72, 3650, 72 & 1 & 2.16e-08 & -1.7827863e+03 & 196.3 \\
 & & SDPNAL+ & 22, 24, 650 & 1 & 7.91e-07 & -1.7827613e+03 & 398.5 \\
\addlinespace
\multirow[t]{2}{*}[0pt]{3} & \multirow[t]{2}{*}[0pt]{\parbox[t]{2cm}{(80,3321)\\m=3590002}} & RPOP & 47, 2400, 47 & 1 & 2.57e-08 & -2.9604567e+03 & 720.0 \\
 & & SDPNAL+ & 22, 27, 550 & 1 & 6.27e-07 & -2.9604002e+03 & 2329.4 \\
\addlinespace
\bottomrule
\end{longtable}
}

Table~\ref{table:NSTF-3} demonstrates that both RPOP and SDPNAL+ solve all 
second-order relaxation problems to the required accuracy, producing rank-one or nearly rank-one solutions. RPOP outperforms SDPNAL+ across all test cases, offering 3–5 times speedup on moderate-size instances and up to 7 times on the largest.

{\small
\begin{longtable}{@{}l l l l r r r r r@{}}
\caption{Computational results for NSTF problems with $\rmt=4$
(relaxation order $\tau=2$).}\label{table:NSTF-4} \\
\toprule
NSTF & {(n, N), m} & \multicolumn{1}{l}{Algorithm} & \multicolumn{1}{l}{Iteration} & \multicolumn{1}{c}{Rank} & \multicolumn{1}{c}{$R_{\max}$} & \multicolumn{1}{c}{Objective} & {Time} \\
\midrule
\endfirsthead
\toprule
NSTF & {(n, N), m} & \multicolumn{1}{l}{Algorithm} & \multicolumn{1}{l}{Iteration} & \multicolumn{1}{c}{Rank} & \multicolumn{1}{c}{$R_{\max}$} & \multicolumn{1}{c}{Objective} & {Time} \\
\midrule
\endhead
\multirow[t]{2}{*}[0pt]{1} & \multirow[t]{2}{*}[0pt]{\parbox[t]{2cm}{(40,861)\\m=236202}} & RPOP & 19, 1000, 19 & 2 & 1.99e-07 & -2.7517484e+02 & 8.5 \\
 & & SDPNAL+ & 16, 16, 800 & 3 & 5.64e-07 & -2.7517493e+02 & 67.0 \\
\addlinespace
\multirow[t]{2}{*}[0pt]{1} & \multirow[t]{2}{*}[0pt]{\parbox[t]{2cm}{(60,1891)\\m=1155402}} & RPOP & 18, 950, 18 & 2 & 9.42e-07 & -4.9982804e+02 & 50.5 \\
 & & SDPNAL+ & 1, 1, 524 & 2 & 6.62e-07 & -4.9982798e+02 & 232.7 \\
\addlinespace
\multirow[t]{2}{*}[0pt]{1} & \multirow[t]{2}{*}[0pt]{\parbox[t]{2cm}{(80,3321)\\m=3590002}} & RPOP & 28, 1450, 28 & 2 & 6.43e-07 & -7.6119587e+02 & 365.0 \\
 & & SDPNAL+ & 11, 11, 1142 & 7 & 5.72e-06$^{\dagger}$ & -7.6123071e+02 & 3603.6 \\
\addlinespace
\multirow[t]{2}{*}[0pt]{2} & \multirow[t]{2}{*}[0pt]{\parbox[t]{2cm}{(40,861)\\m=236202}} & RPOP & 23, 1200, 23 & 2 & 3.65e-07 & -1.0807713e+03 & 10.1 \\
 & & SDPNAL+ & 13, 23, 350 & 2 & 9.88e-07 & -1.0807709e+03 & 41.0 \\
\addlinespace
\multirow[t]{2}{*}[0pt]{2} & \multirow[t]{2}{*}[0pt]{\parbox[t]{2cm}{(60,1891)\\m=1155402}} & RPOP & 21, 1100, 21 & 2 & 8.72e-07 & -2.3932547e+03 & 59.7 \\
 & & SDPNAL+ & 18, 36, 350 & 2 & 8.52e-07 & -2.3932470e+03 & 259.2 \\
\addlinespace
\multirow[t]{2}{*}[0pt]{2} & \multirow[t]{2}{*}[0pt]{\parbox[t]{2cm}{(80,3321)\\m=3590002}} & RPOP & 24, 1250, 24 & 2 & 5.41e-07 & -4.2204249e+03 & 320.7 \\
 & & SDPNAL+ & 23, 46, 550 & 2 & 8.67e-07 & -4.2204101e+03 & 2501.8 \\
\addlinespace
\multirow[t]{2}{*}[0pt]{3} & \multirow[t]{2}{*}[0pt]{\parbox[t]{2cm}{(40,861)\\m=236202}} & RPOP & 22, 1150, 22 & 2 & 1.84e-07 & -6.1939590e+03 & 9.9 \\
 & & SDPNAL+ & 3, 3, 350 & 2 & 4.75e-07 & -6.1939856e+03 & 33.6 \\
\addlinespace
\multirow[t]{2}{*}[0pt]{3} & \multirow[t]{2}{*}[0pt]{\parbox[t]{2cm}{(60,1891)\\m=1155402}} & RPOP & 19, 1000, 19 & 2 & 2.66e-07 & -1.5592573e+04 & 54.7 \\
 & & SDPNAL+ & 7, 7, 350 & 2 & 8.36e-07 & -1.5592509e+04 & 205.1 \\
\addlinespace
\multirow[t]{2}{*}[0pt]{3} & \multirow[t]{2}{*}[0pt]{\parbox[t]{2cm}{(80,3321)\\m=3590002}} & RPOP & 15, 800, 15 & 7 & 8.53e-07 & -2.9869486e+04 & 207.5 \\
 & & SDPNAL+ & 6, 13, 300 & 2 & 9.58e-07 & -2.9869493e+04 & 1146.0 \\
\addlinespace
\bottomrule
\end{longtable}
}

As shown in Table~\ref{table:NSTF-4}, RPOP consistently solves all second-order Poly-SDP-RLT relaxations of the NSTF problem, while SDPNAL+ fails for some large-scale instances. RPOP achieves significant speedups over SDPNAL+ across all data types. For small- to medium-scale problems, RPOP is about 3–6 times faster, while for the largest instances ($n=80$), it can be up to 10 times faster.

\section{Conclusions}\label{sec-conclusion}

In this paper, we proposed a unified framework for solving polynomial optimization problems (POPs) via polyhedral–SDP relaxations. We focused primarily on problems with equality and nonnegativity constraints, and showed that general inequality constraints can be equivalently reformulated into this setting using standard or squared slack variables. Within this framework, we established connections and analyzed the relative tightness among various relaxation schemes, including DNN, RLT, and moment–SOS.
To efficiently solve the resulting large-scale relaxations, we developed RiNNAL-POP, a low-rank augmented Lagrangian method that alternates between a low-rank phase and a convex lifting phase. The algorithm incorporates two key features that enable scalability and efficiency. First, the large numbers of nonnegativity inequality constraints and consistency equality constraints—both typically of order $\Omega(n^{2\tau})$—are handled efficiently and jointly through a tailored projection procedure in the convex lifting phase. Second, we identify and exploit the hidden facial structure present in many relaxations and incorporate this into the algorithmic design. In particular, after applying the Burer–Monteiro factorization, we restrict the low-rank subproblems to affine subspaces corresponding to exposed faces of the semidefinite cone, thereby improving numerical stability and computational efficiency. Moreover, our method automatically adjusts the factorization rank, and we provide a theoretical guarantee, via rank identification, that the final rank coincides with that of an optimal solution.
Numerical experiments on a variety of test problems demonstrate the robustness and scalability of the proposed method. 
A natural continuation of the current work is to integrate this framework with relaxation hierarchies to develop efficient global solvers for exactly solving general POPs, particularly those involving nonnegativity constraints.

\section*{Acknowledgments}
The authors express their gratitude to Professor Ching-pei Lee for his valuable comments on the rank identification results.

\bibliographystyle{abbrv}
\bibliography{refs}


\appendix
\section{Examples} 
\label{sec-A}
\begin{example}\label{example-pop}
    Consider the following quadratic optimization problem in the form of \eqref{eq-POP}:
    \begin{equation*}
        \zeta^*=\min \Big\{ f_0(w) \mid  f_i(w) = 0,\; i\in [2],\; w\in \R^3_+ \Big\},
    \end{equation*}
    where the objective and constraint functions are defined as
    \begin{align*}
        & f_0(w)=w_1^2+4w_1-w_3^2,\quad
         f_1(w)=(w_2+w_3-1)^2,\quad
         f_2(w)=(w_1-2)^2+w_2(2w_1-3).
    \end{align*}
    Take degree $\tau=1$ and select the monomial basis with its index set
    \[
    \cA=\cA_1=\left\{
    \begin{bmatrix}
        1\\0\\0\\0
    \end{bmatrix},
    \begin{bmatrix}
        0\\1\\0\\0
    \end{bmatrix},
    \begin{bmatrix}
        0\\0\\1\\0
    \end{bmatrix},
    \begin{bmatrix}
        0\\0\\0\\1
    \end{bmatrix}
    \right\},\quad 
    u^\cA (x)=x=\begin{bmatrix}
        x_0\\x_1\\x_2\\x_3
    \end{bmatrix}
    =\begin{bmatrix}
        1\\w_1\\w_2\\w_3
    \end{bmatrix}.
    \]
    The lifted matrix variable is given by
    \[
    X = u^{\cA}(x)\,u^{\cA}(x)^\top
    = \begin{bmatrix}
    x_0^2 & x_0x_1 & x_0x_2 & x_0x_3\\
    x_1x_0 & x_1^2 & x_1x_2 & x_1x_3\\
    x_2x_0 & x_2x_1 & x_2^2 & x_2x_3\\
    x_3x_0 & x_3x_1 & x_3x_2 & x_3^2
    \end{bmatrix}
    = \begin{bmatrix}
    1 & w_1 & w_2 & w_3\\
    w_1 & w_1^2 & w_1w_2 & w_1w_3\\
    w_2 & w_2w_1 & w_2^2 & w_2w_3\\
    w_3 & w_3w_1 & w_3w_2 & w_3^2
    \end{bmatrix}.
    \]
    Since $D=\R^3_+$ is the nonnegative orthant, we obtain the following convex CPP relaxation:
    \begin{equation*}
     \zeta^J =\min\Big\{\inprod{Q^0}{X} \mid X \in  J, \; \inprod{H^0}{X}=1\Big\}.
    \end{equation*}
    where $J=\{X\in{\rm CPP}^4\mid \inprod{Q^i}{X} = 0,\;i=1,2\}$, 
    $H^0\in \R^{4\times 4}$ has all its entries equal to zero except
    having the entry one at its $(1,1)$ position, 
    and 
    \begin{align*}
    Q^0&=\begin{bmatrix}
            0 & 2 & 0 & 0\\
            2 & 1 & 0 & 0\\
            0 & 0 & 0 & 0\\
            0 & 0 & 0 & -1
        \end{bmatrix},\quad
    Q^1 =A^\top A=
\begin{bmatrix} -1 \\ 0 \\ 1 \\ 1 \end{bmatrix}
\begin{bmatrix} -1 \\ 0 \\ 1 \\ 1 \end{bmatrix}^\top
,\quad
        Q^2&=\begin{bmatrix}
            4 & -2 & -3/2 & 0\\
            -2 & 1 & 1 & 0\\
            -3/2 & 1 & 0 & 0\\
            0 & 0 & 0 & 0
        \end{bmatrix}.
    \end{align*}
As shown in \cite[Example 4.9]{kim2020geometrical}, this convex reformulation is exact, i.e., $\zeta^{J} = \zeta^*$. By taking $\mathcal{P}^{\mathcal{A}} = \bN^4$, we obtain the following DNN relaxation:
$$
\zeta^{\mathrm{DNN}}
= \min 
\Bigl\{\langle Q^0, X\rangle \;\Big|\;
\langle H^0, X\rangle = 1,\;
A X = 0,\;
\langle Q^2, X\rangle = 0,\;
X\in \mathbb{S}_+^4 \cap \bN^4
\Bigr\},
$$
where $\langle Q^1, X\rangle = 0$ is replaced by its equivalent constraint $A X = 0$ (due to the decomposition
$Q^1 = A^\top A$ and positive semidefinitenss of $X$).
In this low-order case, the consistency cone $\mathcal{L}^{\mathcal{A}}$ coincides with $\mathbb{S}^{\mathcal{A}}$ and is therefore omitted.
\end{example}

\end{document}